\documentclass[reqno, 12pt]{article}

\pdfoutput=1

\usepackage[linesnumbered,ruled,vlined]{algorithm2e}
\usepackage{array}
\usepackage[centertags]{amsmath}
\usepackage{amsfonts}
\usepackage{amssymb}
\usepackage{amsthm}
\usepackage{booktabs}
\usepackage{caption}
\usepackage{color}
\usepackage{diagbox}
\usepackage{enumerate}
\usepackage{extarrows}
\usepackage{float}
\usepackage{graphics}
\usepackage{hyperref}
\usepackage{latexsym}
\usepackage{longtable}
\usepackage{multirow}
\usepackage{newlfont}
\usepackage[sort,compress,numbers]{natbib}
\usepackage{rotating}
\usepackage{seqsplit}
\usepackage{tocloft}
\usepackage{titlesec}
\usepackage[utf8]{inputenc}

\newtheorem{thm}{Theorem}[section]
\newtheorem{cor}[thm]{Corollary}
\newtheorem{lem}[thm]{Lemma}
\newtheorem{prop}[thm]{Proposition}

\theoremstyle{mydefinition}
\newtheorem{dfn}[thm]{Definition}
\theoremstyle{myremark}
\newtheorem{rem}[thm]{Remark}
\newtheorem{exa}[thm]{Example}

\newtheorem{alg}[thm]{Algorithm}

\allowdisplaybreaks[4]
\SetKwInput{KwInput}{Input}
\SetKwInput{KwOutput}{Output}

\def\N{\mathbb{N}}
\def\a{\mathbf{a}}
\def\b{\mathbf{b}}
\def\0{\mathbf{0}}
\def\1{\mathbf{1}}
\def\Z{\mathbb{Z}}
\def\CT{\mathop{\mathrm{CT}}}
\def\y{\mathbf{y}}
\def\sgn{\mathrm{sgn}}

\newcommand{\C}{{\mathbb{C}}}

\renewcommand{\P}{{\mathbb{P}}}

\newcommand{\R}{{\mathbb{R}}}

\newcommand{\Q}{{\mathbb{Q}}}

\title{Arithmetic Complexity of Evaluating Aardal-Lenstra Denumerants via Constant Term Method}

\author{Jinlong Tang$^{\color{blue} \dag}$, Guoce Xin$^{\color{blue} \S}$, and Zihao Zhang$^{\color{blue} \P}$
\\[2mm]
{\small $^{\color{blue} \dag, \S}$ School of Mathematical Sciences,}\\[-0.8ex]
{\small Capital Normal University, Beijing, 100048, P.R.~China}\\
{\small $^{\color{blue} \P}$ School of Mathematics and Statistics,}\\[-0.8ex]
{\small Beijing Institute of Technology, Beijing, 102400, P.R.~China}\\
{\small {\color{blue} $^\dag$} Email address: jinlong\_tang@cnu.edu.cn}\\
{\small {\color{blue} $^\S$} Email address: guoce\_xin@163.com}\\
{\small {\color{blue} $^\P$} Email address: zihao-zhang@foxmail.com}
}

\date{\today}

\begin{document}

\maketitle

\begin{abstract}
Aardal and Lenstra systematically investigated hard knapsack problems of the form $a_1x_1+\cdots+a_nx_n=b$, where the weights are parameterized as $a_i=p_iM+r_iN$. Here, $M$ and $N$ are coprime positive integers, and the magnitudes $\vert{}p_i\vert{}$ and $\vert{}r_i\vert{}$ are small relative to $M$ and $N$. We investigate the corresponding challenging denumerant problem---namely, counting the number of nonnegative integer solutions---for a given representation of the weights. After one extended Euclidean computation on $M,N$, the algorithm uses $\widetilde O(n^4\Delta^2)$ arithmetic operations, where $\Delta=\max_{i,j}\vert{}r_i p_j-r_j p_i\vert{}$ and $\widetilde O$ suppresses logarithmic factors. If an integer row $\mathbf{k}$ is given such that $\mathbf{a}$ and $\mathbf{k}$ are linearly independent and $|k_i a_j-k_j a_i|\leq\Delta$ for all $i,j$, we may take $M=N=1$ and compute the denumerant in a total of $\widetilde O(n^4\Delta^2)$ arithmetic operations. For positive weights bounded by $D$, the denumerant can be computed using $\widetilde O(n^3D^2)$ arithmetic operations. The same bound holds when one weight is arbitrarily large and the other $n-1$ weights are bounded by $D$. Moreover, we consider the problem of expressing a general sequence $(a_1,\dots,a_n)$ in the above form using the LLL algorithm.
\end{abstract}

\noindent
\begin{small}
\end{small}

\noindent
\begin{small}
\end{small}

\section{Introduction}
This paper studies the computation of the denumerant $d(b;a_1,\dots,a_n)$,
which counts the number of nonnegative integer solutions to the linear
Diophantine equation $a_1 x_1 + \cdots + a_n x_n = b$.
Throughout, we write $\mathbf{a} = (a_1, a_2, \dots, a_n)$ for a sequence of
positive integers with $n\geq 2$, and we assume $b \in \N$ to be a nonnegative integer and $\gcd(\mathbf{a}) = 1$
unless specified otherwise.

It is well known that
\begin{equation}
d(b;\mathbf{a}) = \CT_q \frac{q^{-b}}{\prod_{i=1}^{n}(1 - q^{a_i})},
\label{equ-denumerant}
\end{equation}
where $\CT\limits_q$ extracts the constant term of a Laurent series in $q$.
Geometrically, $d(b;\mathbf{a})$ also equals the number of lattice points in
the $(n-1)$-dimensional rational convex polytope
\[
P(b;\mathbf{a}) = \{\, \mathbf{x} \in \R^n :
\mathbf{a} \cdot \mathbf{x} = b,\ \mathbf{x} \ge \mathbf{0} \,\}.
\]

Computing $d(b;\mathbf{a})$ is $\#P$-hard, and even deciding whether
$d(b;\mathbf{a}) > 0$ is notoriously NP-complete.
In computational geometry, two polynomial-time lattice-point counting
algorithms are available for special classes of polytopes.
The first is Barvinok's algorithm~\cite{barvinok1994polynomial}, which runs
in polynomial time when the dimension is fixed.
The second is the $\Delta$-modular algorithm of Gribanov and
Zolotykh~\cite{GribanovZolotykh2022}, which runs in polynomial time when
both the parameter $\Delta$ and the number of additional constraints are fixed.
Applied to the polytope $P(b;\mathbf{a})$, Barvinok's algorithm requires the
dimension $n-1$ to be fixed, whereas the $\Delta$-modular algorithm requires
each coefficient $a_i$ to be bounded by a constant $\Delta$
(i.e., $\max\limits_i a_i \le \Delta$).

Our work is motivated by the hard knapsack instances studied by
Aardal and Lenstra~\cite{AL04}.
Modern solvers handle the classical knapsack problem (finding an integer
solution to $\sum_{i=1}^n a_i x_i = b$, if one exists) extremely well in
general, but they struggle on specially crafted configurations.
In these hard instances the coefficients take the form
\[
a_i = p_i M + r_i N \in \mathbb{Z}_{+},
\]
where $M$ and $N$ are extremely large compared with $p_i \in \mathbb{N}$
and $|r_i|$, and the right-hand side $b$ is chosen near the Frobenius
number of $\mathbf{a}$ (the largest integer $g$ with $d(g;\mathbf{a}) = 0$).
The resulting polytope $P(b;\mathbf{a})$ is then extremely flat, creating a
huge integrality gap that severely degrades the performance of standard
branch-and-bound algorithms.
Aardal and Lenstra elegantly overcame this difficulty for the feasibility
problem by reformulating it in an LLL-reduced lattice basis, which
dramatically prunes the search tree.

These challenging instances have since been studied extensively through
Barvinok-type methods, decomposable reformulations, and constant term
techniques~\cite{DeLoeraHawsHemmeckeHugginsYoshida2005,
KrishnamoorthyPataki2009, AardalScavuzzoWolsey2023, XZZ25,
DeLoeraHemmeckeTauzerYoshida2004, BBDDKV15, XinZhang2023,
xin2024combinatorial}.
However, although current algorithms can compute the associated denumerants
efficiently for $n \le 10$, their running times become prohibitive for larger
instances, falling far short of the efficiency that lattice reformulations
achieve for the feasibility problem. To break through this persistent computational bottleneck, we propose a new
approach to denumerant computation from the viewpoint of algebraic
combinatorics.
Specifically, by leveraging constant term manipulations, we establish  algorithms for specific families of denumerants with a polynomial number of arithmetic operations over $\Q$. We write $\widetilde O$ to suppress logarithmic factors in those parameters.

For a positive integer sequence $\mathbf{a}$, a representation of the form
$\mathbf{a} = \mathbf{p}M + \mathbf{r}N$ is called an \emph{A-L representation}
(named after Aardal and Lenstra), provided that
$\mathbf{p}=(p_{1},\dots,p_{n})$ and $\mathbf{r}=(r_{1},\dots,r_{n})$ are
linearly independent integer vectors, and $M,N$ are coprime positive integers.
We say that $\mathbf{a}$ is a \emph{$B$-bounded A-L sequence} if it admits an
A-L representation with $|p_i|,|r_i| \le B$ for all $i$.
Similarly, $\mathbf{a}$ is a \emph{$\Delta$-modular A-L sequence} if it admits
an A-L representation such that the matrix
$\binom{\mathbf{p}}{\mathbf{r}}$ is $\Delta$-modular;
i.e., $|r_i p_j - r_j p_i| \le \Delta$ for all $i,j$.
The corresponding denumerants are referred to as
\emph{$B$-bounded A-L denumerants} and
\emph{$\Delta$-modular A-L denumerants}, respectively. Note that a $B$-bounded A-L sequence is automatically a
$2B^{2}$-modular A-L sequence, so the $\Delta$-modular condition defines a
broader class of problems.

Our main contribution is an arithmetic complexity bound for computing denumerants in each of these classes when the A-L representation is given.
Moreover, for a given $\mathbf{a}$, we use the LLL algorithm of Lenstra, Lenstra, and Lov\'asz~\cite{LLL82} to find an A-L
representation and to approximate
the minimal possible $\Delta$ for which $\mathbf{a}$ is $\Delta$-modular.
Naturally, our algorithms become inefficient when this minimal $\Delta$ is large.

The remainder of this paper is organized as follows. Section~\ref{pre} reviews existing lattice-point counting methods and outlines two fundamental tools. In Section~\ref{sec-3}, the input consists of a nonnegative integer $b$ and positive weights $a_1,\ldots,a_n$ with $\gcd(\mathbf a)=1$. When all $n$ weights are bounded by $\Delta$, Proposition~\ref{prop-polymialwhenbounded} gives an algorithm that computes $d(b;\mathbf a)$ using $\widetilde O(n^3\Delta^2)$ arithmetic operations. Theorem~\ref{thm-onlyonetermbig} gives the same bound when only $a_1,\ldots,a_{n-1}$ are bounded by $\Delta$ and $a_n$ is arbitrarily large.

The algorithms in Sections~\ref{sec-Bbounded} and~\ref{sec-5} take as input $b$ and an A-L representation $(M,N,\mathbf p,\mathbf r)$ satisfying $a_i=p_iM+r_iN$. Both algorithms start with one extended Euclidean computation on $M,N$, costing $O(\log(2+\min\{M,N\}))$ arithmetic operations. If $|p_i|,|r_i|\le B$ for all $i$, the algorithm of Theorem~\ref{thm-crudealgocomplexity} computes $d(b;\mathbf a)$ in $\widetilde O(n^4B^6)$ further operations by reduction to nearly bounded denumerants. Using the matrix operations of Xin, Xu, and Zhang~\cite{XXZ25}, Theorem~\ref{thm-PolyAL} shows that \texttt{PolyAL} needs only $\widetilde O(n^4\Delta^2)$ operations after the same initial step when $|r_ip_j-r_jp_i|\le\Delta$ for all $i,j$. Subsection~\ref{subsec-LLLHKP} considers how to find a representation with small minors when only the weights are given. Finally, Section~\ref{remark} offers concluding remarks.

\section{Preliminaries}\label{pre}

In this section, we review existing approaches to the knapsack problem and outline two key tools central to our methodology, both developed by Xin et al.~\cite{X04,X15,XXZ25,XZ25,XZZ25}: an efficient algorithm for computing generalized Todd polynomials and the constant term method.

\subsection{Computation scheme}\label{subsec-scheme}
Consider the polytope $P = P(b;\a) = \{\mathbf{x} \in \R^{n} : \a\mathbf{x} = b, \mathbf{x} \geq \mathbf{0}\}$, where $\a = (a_{1}, a_{2}, \dots, a_{n})$ is a sequence of positive integers and $b \in \Z_{+}$. The primary objective is to evaluate the associated counting function $\sigma_P(\mathbf{1})$.

Counting lattice points within a rational convex polytope $P \subset \mathbb{R}^n$ is a fundamental problem in computational geometry. Efficient algorithms for this task (e.g., \texttt{LattE} \cite{DeLoeraHemmeckeTauzerYoshida2004}) typically proceed in two major steps.

The first major step is to obtain a rational encoding of the lattice point generating polynomial $\sigma_P(\y)$.
In computational geometry, this is usually done through several steps to write $\sigma_P(\y)$ in the form
$$ \sigma_P(\y)=\sum_{\alpha\in P\cap \Z^n}\y^{\alpha}   = \sum_{i\in I}\epsilon_{i}\frac{\y^{p_{i}}}{(1-\y^{a_{i1}})(1-\y^{a_{i2}})\cdots(1-\y^{a_{in}})} ,$$
where $\epsilon_{i}\in\{-1,+1\}$ and $|I|$ can be bounded by a polynomial in the input size when the dimension of $P$ is fixed.

The second primary step involves computing the limit of $\sigma_P(\y)$ as $y_i \to 1$ for all $i$. This process requires substituting $y_i=e^{\kappa_i s}$ using an \emph{admissible} integral vector $\kappa$, and subsequently extracting the constant term. By \emph{admissible}, it is meant that the substitution prevents any denominator from evaluating to zero; as demonstrated in \cite{BP99}, such a vector can be found in polynomial time. This step can be treated as an independent algebra computation, called the Gtodd computation, and will be elaborated in the next subsection.

Our methodology builds upon Xin's constant term approach in algebraic combinatorics. For the polytope $P=P(b;\a)$, we express $\sigma_P(\y)$ as a related yet distinct sum of simple rational functions, characterized by two key features: (i) the denominator factors may involve roots of unity; (ii) the number of summands is significantly smaller for certain special sequences $\a$. The second feature makes our representation more efficient for applications in integer
programming, discrete optimization, and symbolic summation
\cite{DeLoeraHemmeckeTauzerYoshida2004,DHK13}.

\subsection{Generalized Todd polynomials}
In this subsection, we introduce the fast computation
of generalized Todd polynomials developed by Xin et al.
\cite{XZZ25}, which makes use of the log-exponential
trick. To do this, we will always use the notation $R_{f}=\Q(\zeta)$, where $\zeta=e^{\frac{2\pi\imath}{f}}$ is a primitive $f$-th root of unity. Note that $R_{f}$ reduces to $\Q$ when $\zeta=1$. This notation also applies to cases with multiple roots of unity, since $\Q\left(e^{\frac{2\pi\imath}{f_1}}, e^{\frac{2\pi\imath}{f_2}}\right)= \Q\left(e^{\frac{2\pi\imath}{f}}\right)$, where $f$ is the least common multiple (lcm) of $f_1$ and $f_2$.

The cyclotomic field $\Q(\zeta)$ is isomorphic to $\Q[x]/\langle \Phi_f(x) \rangle$, where $\Phi_f(x)$ is the $f$-th cyclotomic polynomial. For convenience, we measure the algorithm's complexity in terms of arithmetic operations in $R_{f}$. Using the Fast Fourier Transform (FFT) \cite{V13}, each operation in $R_{f}$ takes $O\left(\phi(f)\log\left(\phi(f)\right)\right)$ operations in $\Q$, where $\phi$ is Euler's totient function.

Let $a\in \Q$, $z_1,\dots,z_r$ be auxiliary parameters and let $B_0,\bar B_0,B_1,\bar B_1,\dots, B_r, \bar B_r$ be finite multisets of nonzero rational numbers. Suppose their cardinalities are $k_i$ and $\bar{k}_i$, respectively, for $i = 0, \dots, r$.  The generalized Todd polynomials, hereafter referred to as the Gtodd polynomials, $$gtd_n:=gtd_n(a,B_0,\bar B_0,B_1,\bar B_1,\dots, B_r, \bar B_r)$$ are defined by their generating function
\begin{align}
  F(s)=\sum_{n\ge 0} gtd  _n s^n&=e^{as} \frac{ \prod_{b\in B_0} g(bs)}{\prod_{b\in \bar B_0} g(b s)}
  \prod_{i=1}^r  \frac{\prod_{b\in B_i} g(bs,z_i)} {\prod_{b\in \bar B_i} g(bs,z_i)},
\end{align}
where
\begin{align*}
  g(s)= \frac{s}{e^s-1}=1+o(1), \qquad  g(s,z)= \frac{1}{1-z (e^s-1)}=1+o(1).
\end{align*}

We are primarily interested in the case where $\bar B_i$ is empty and $z_i\in R_{f}$ for all $i$.
Then, the generating function simplifies to
\begin{align}
  F(s)=\sum_{n\ge 0} gtd_n s^n&=e^{as} \prod_{b\in B_0} g(bs)
  \prod_{i=1}^r  \prod_{b\in B_i} g(bs,z_i)\notag\\
  &=e^{as} \prod_{b\in B_0} \frac{bs}{e^{bs}-1}
  \prod_{i=1}^r  \prod_{b\in B_i}\frac{1}{1-z_i(e^{bs}-1)}\label{e-GTodd}.
\end{align}

The log-exponential trick consists of first computing $H(s)=\ln(F(s))$ via the expansion formulas
\begin{align*}
  h(s)=\ln g(s)=-\sum_{k\geq1}\frac{\mathcal{B}_{k}}{k\cdot k!}s^{k}, \text{\quad} h(s,z)=\ln g(s,z)=\sum_{k\geq1 }C_{k}(z)s^{k},
\end{align*}
and then evaluating $e^{H(s)}$. Here, the coefficients $\mathcal{B}_{k}$ are the classical Bernoulli numbers, and $C_{k}(z)$ are polynomials in $z$. Following this setup, Xin and Zhang \cite{XZ25} provided the following complexity result for the case where $z_{i}\in R_{f}$ for all $i$:

\begin{thm}[Xin and C.~Zhang~{\cite[Theorem~2.6]{XZ25}}]\label{theo-GTodd}
Suppose $F(s)$ is given by \eqref{e-GTodd}, where $a \in \Q$, the collections $B_i$ are finite multisets of rational numbers, and $z_i \in R_f$ for all $i$. For any given positive integer $d$, we can evaluate the first $d$ generalized Todd polynomials $(gtd_0, gtd_1, \dots, gtd_{d-1})$ using $O\left( (r+1)d\log(d) + \log^2(d)\sum_{i=0}^{r} |B_i| \right)$ operations in $R_f$, which corresponds to
$$O\left( \left( (r+1)d\log(d) + \log^2(d)\sum_{i=0}^{r} |B_i| \right) \phi(f)\log(\phi(f)) \right)$$
operations in $\Q$.
\end{thm}

\subsection{ Brief introduction to the XinPF algorithm}
In this subsection, we briefly review Xin's partial fraction
algorithm (XinPF).
The variable order $\texttt{vars}=[y_1,\dots,y_n]$ defines
the field of iterated Laurent series
\cite{X04}:
\begin{align}
  G_1=\mathbb{C}((y_n))((y_{n-1}))\cdots((y_1)).
\end{align}
To allow fractional powers of the variables, we work in the
extended field
\begin{align}
  G=\bigcup_{D\geq 1} G_D,\qquad
  G_D=\mathbb{C}((y_n^{1/D}))((y_{n-1}^{1/D}))
  \cdots((y_1^{1/D})),
  \label{Ite-Laurent}
\end{align}
where $D$ ranges over the positive integers and
$G_D\subseteq G_{D'}$ whenever $D\mid D'$,
under the natural embeddings. 

A monomial $M=\zeta\y^\alpha$, where $\zeta\in\C$
is a root of unity and
$\alpha\in\Q^n\setminus\{\mathbf{0}\}$,
is \emph{small} ($M<1$) if the first nonzero entry
of $\alpha$ is positive, and \emph{large} ($M>1$)
if it is negative.
In $G$, we utilize the unique series expansions:
\begin{align} \label{issmall}
  \frac{1}{1-M} =\left\{
                   \begin{array}{ll}
                    \displaystyle \sum_{k\ge 0} M^k, & \text{ if } M<1; \\
                    \displaystyle \frac{1}{-M(1-1/M)}=  \sum_{k\ge 0} - \frac{1}{M^{k+1}}, & \text{ if } M>1.
                   \end{array}
                 \right.
\end{align}

Now we work in the extended field $G$ associated with
the variable order
$\texttt{vars}=[\y,\Lambda]
=[y_1,\dots,y_n,\lambda_1,\dots,\lambda_r]$. For any $F\in G$, define the constant term operator
$\CT_\Lambda$ by
\[
  \CT_\Lambda F
  =\CT_\Lambda
    \sum_{\mathbf{i}} f_{\mathbf{i}}
    \lambda_1^{i_1}\cdots\lambda_r^{i_r}
  =f_{\mathbf{0}},
\]
where $\mathbf{i}$ ranges over $(D^{-1}\Z)^r$
for some positive integer $D$.
Since the individual constant term operators
$\CT_{\lambda_i}$ commute with one another,
$\CT_\Lambda$ unambiguously extracts the constant
term with respect to the entire set of variables $\Lambda$.

To be precise, we first write $E$ as a \emph{standard rational function} in $\lambda$:
\begin{align}\label{e-positive-form}
E= \frac{L(\lambda)}{\prod_{i=1}^n (1-u_i \lambda^{a_i})},
\end{align}
where $L(\lambda)$ is a Laurent polynomial, $u_i$ are independent of $\lambda$ and $a_i$ are positive integers for all $i$. For a standard rational function $F(\lambda)$, $\deg_\lambda F(\lambda)$ denotes the degree of its numerator minus the degree of its denominator.
$F(\lambda)$ is termed \emph{proper} if $\deg_\lambda F(\lambda) < 0$.

\begin{prop}[Xin~{\cite[Theorem~8 and Lemma~9]{X15}}]\label{p-partialfraction}
Suppose that the factors $1-u_i\lambda^{a_i}$ are pairwise coprime and $E(\lambda)$ admits a partial fraction decomposition of the form
\begin{align*}
E= P(\lambda)+ \frac{p(\lambda)}{\lambda^k} +\sum_{i=1}^n \frac{A_i(\lambda)}{1-u_i \lambda^{a_i}},
\end{align*}
where the $u_i$'s are independent of $\lambda$, $P(\lambda)$, $p(\lambda)$, and the $A_i(\lambda)$'s are all polynomials, $\deg p(\lambda)<k$, and $\deg A_i(\lambda)<a_i$ for all $i$.
Then $A_{i}(\lambda)$ is uniquely characterized by $A_{i}(\lambda) \equiv E(\lambda)(1-u_{i}\lambda^{a_{i}}) \pmod{\langle 1-u_{i}\lambda^{a_{i}}\rangle}$ with $\deg A_{i}<a_{i}$.
If $E$ is proper in $\lambda$, i.e., $\deg_\lambda E < 0$, then
\begin{equation}\label{equ-proper}
\CT_\lambda E = \sum_{u_i \lambda^{a_i} <1} A_i(0).
\end{equation}
Moreover, if $\deg_\lambda E < 0$ and $E|_{\lambda=0}=0$, we deduce that
\begin{equation}\label{equ-properandlim=0}
A_{1}(0)+A_{2}(0)+\dots+A_{n}(0)=0.
\end{equation}
\end{prop}

\begin{dfn}[following Xin, Xu, and Zhang~{\cite[Definition~2]{XXZ25}}]\label{defn-A}Let $E$ be defined as above. We introduce the notation
$$\CT_\lambda \frac{1}{\underline{1-u_s \lambda^{a_s}}} \cdot E(1-u_s \lambda^{a_s})=\CT_\lambda \frac{1}{\underline{1-(u_s \lambda^{a_s})^{-1}}}\cdot    E(1-(u_s \lambda^{a_s})^{-1}) :=A_s(0),$$
stipulating that this term is $0$ whenever $a_s=0$. This operation extends additively over products of denominator factors.
\end{dfn}

By allowing fractional powers and roots of unity, this extraction admits the following explicit formulation.

\begin{lem}[Xin, Xu, and Zhang~{\cite[Lemma~3]{XXZ25}}]\label{lem-unit-E}
Suppose $E$ in \eqref{e-positive-form} is rewritten such that $b_i=\pm a_i$ for all $i$. Then for any nonzero integer $b_s$,
\begin{equation}\label{equ-addifomula-A}
\CT_\lambda \frac{1}{\underline{1-u_s \lambda^{b_s}}}\cdot E (1-u_s \lambda^{b_s})=\frac{1}{b_s}\sum_{\zeta: \zeta^{b_s}=1}E (1-u_s \lambda^{b_s}) \Big|_{\lambda=({u_s}^{\frac{1}{b_s}}\zeta)^{-1}}.
\end{equation}
\end{lem}

\section{Arithmetic complexity of nearly \texorpdfstring{$\Delta$}{Delta} bounded denumerants}\label{sec-3}
In this section, we first present a new proof that $d(b;\a)$ can be computed using a number of arithmetic operations polynomial in $n$ and $\Delta$ when $\a$ is a \emph{$\Delta$-bounded sequence}, meaning $a_i \leq \Delta$ for all $i$. We then extend this result to a \emph{nearly $\Delta$-bounded sequence}, where we allow exactly one entry, say $a_n$, to be greater than $\Delta$. The corresponding denumerant is referred to as a \emph{(nearly) $\Delta$-bounded denumerant}. We will use this new result in the next section for $B$-bounded A-L denumerants.

The computation of $\Delta$-bounded denumerants has been
studied in the broader context of lattice-point counting
in $\Delta$-modular polyhedra. Gribanov and Zolotykh
\cite{GribanovZolotykh2022} established polynomial-time
computability for fixed $\Delta$. For a given right-hand
side $b$, Gribanov et al.~\cite{GSM2023} obtained an
expected arithmetic complexity bound of
$\widetilde{O}(n^6\Delta^4)$, which was subsequently
improved to $O(n^6\Delta^3)$ by Gribanov et al.
\cite{GSMZ2024}, and to $O(n^5\Delta^3)$ by Dakhno et al.
\cite{DGKKKKM2026}.

The following result was also established in \cite{XZ25} using residue computation with a tighter complexity bound. In contrast, we offer a simplified constant term approach, which leads to a coarser complexity bound. The main purpose of this approach is to establish a
computational framework that can be extended to nearly
$\Delta$-bounded denumerants and, subsequently,
to A-L denumerants.

\begin{prop}\label{prop-polymialwhenbounded}
Let $\a=(a_1,\dots,a_n)$ be a $\Delta$-bounded sequence of positive integers with $\gcd(\a)=1$. Then for any nonnegative integer $b$, the denumerant $d(b;\a)$ can be computed using $O(n^3 \Delta^2 \log n \log \Delta)$ operations in $\Q$.
\end{prop}
\begin{proof}
We start with
\begin{align*}
\sigma_{P(b;\a)}=\CT_{q}\frac{q^{-b}}{(1-q^{a_{1}}y_{1})\cdots(1-q^{a_{n}}y_{n})}, \quad d(b;\a)=\sigma_{P(b;\a)}(\mathbf{1}).
\end{align*}
By Proposition \ref{p-partialfraction}, Definition \ref{defn-A} and Lemma \ref{lem-unit-E}, we obtain
\begin{align*}
\sigma_{P(b;\a)}=&\sum_{i=1}^{n}\CT_{q}\frac{1}{\underline{1-q^{a_{i}}y_{i}}}\frac{q^{-b}}{\prod_{j\neq i}(1-q^{a_{j}}y_{j})}
=&\sum_{i=1}^{n}\left(\frac{1}{a_{i}}\sum_{\zeta:\zeta^{a_{i}}=1}\left.\frac{q^{-b}}{\prod_{j\neq i}(1-q^{a_{j}}y_{j})}\right|_{q=(y_{i}^{\frac{1}{a_{i}}}\zeta)^{-1}}\right).
\end{align*}
This yields a sum of $a_1+\cdots +a_n\leq n\Delta$ rational functions in $y_{i}$, which involve various roots of unity $\zeta$.

Following the computation scheme in Subsection \ref{subsec-scheme}, we first find an admissible vector $\kappa$ and then compute
$$ \CT_s \sigma_{P(b;\a)}\left(e^{\kappa_1s},\dots,e^{\kappa_ns}\right) = \sum \CT_s Q_{i,\zeta}, $$
where the sum runs over all pairs $(i,\zeta)$ such that $\zeta$ is an $a_i$-th root of unity.

A particular constant term is computed as follows.
\begin{align*}
  \CT_s Q_{i,\zeta}&=\CT_s \frac{1}{a_{i}}\frac{e^{\frac{b\kappa_{i}}{a_{i}}s}\zeta^{b}}{\prod_{j\neq i}\left(1-e^{\left(-\frac{a_{j}}{a_{i}}\kappa_{i}+\kappa_{j}\right)s}\zeta^{-a_{j}}\right)}\\
  &=\CT_s \frac{e^{\frac{b\kappa_{i}}{a_{i}}s}\zeta^{b}}{a_{i}}\prod\limits_{j\neq i,\zeta^{a_{j}}=1}\frac{1}{1-e^{\kappa_{i,j}s}}\prod\limits_{j\neq i,\zeta^{a_{j}}\neq1}\frac{1}{1-e^{\kappa_{i,j}s}\zeta^{-a_{j}}}\\
  &=\frac{\zeta^{b}}{a_{i}\prod\limits_{j\neq i,\zeta^{a_{j}}=1}(-\kappa_{i,j})\prod\limits_{j\neq i,\zeta^{a_{j}}\neq1}(1-\zeta^{-a_{j}})}\cdot \CT_s s^{-\tau}F(s)\\
  &=\frac{\zeta^{b}}{a_{i}\prod\limits_{j\neq i,\zeta^{a_{j}}=1}(-\kappa_{i,j})\prod\limits_{j\neq i,\zeta^{a_{j}}\neq1}(1-\zeta^{-a_{j}})}\cdot gtd_{\tau} ,
\end{align*}
where $F(s)$ is a Gtodd generating function explicitly given by
\begin{align}
F(s)=  e^{\frac{b\kappa_{i}}{a_{i}}s}\prod\limits_{j\neq i,\zeta^{a_{j}}=1}\frac{\kappa_{i,j}s}{e^{\kappa_{i,j}s}-1}\prod\limits_{j\neq i,\zeta^{a_{j}}\neq1}\frac{1}{1-\frac{\zeta^{-a_{j}}}{1-\zeta^{-a_j}}\left(e^{\kappa_{i,j}s}-1\right)}
\end{align}
with $\kappa_{i,j}:=-\frac{a_{j}}{a_{i}}\kappa_{i}+\kappa_{j}$ and $\tau:=\#\{1\leq j\leq n:j\neq i,\zeta^{a_j}=1\}$.

Note that $\frac{\zeta^{-a_{j}}}{1-\zeta^{-a_j}}\in \Q(\zeta)\subseteq R_{a_{i}}$.
Next, we apply Theorem \ref{theo-GTodd} with $f=a_{i}, d=\tau+1$, together with the inequalities $d+r\leq n, d\leq n$, and $\sum_{i=0}^{r}|B_{i}|=n-1$. Since $(r+1)d\log(d)\leq (n-d+1)d\log(d)\leq \frac{1}{4}(n+1)^{2}\log(n)$ and $\phi(a_{i})\leq a_{i}\leq \Delta$, evaluating a single term takes at most
$$O\left(\left(\frac{1}{4}(n+1)^{2}\log n+(n-1)\log^{2}n\right)\Delta\log \Delta\right)=O(n^2 \Delta\log n\log \Delta)$$ operations in $\Q$.
The proposition then follows since there are at most $n\Delta$ terms.

This completes the proof.
\end{proof}

The preceding proposition shows that a fast algorithm exists when the entire sequence is bounded. Using basic constant term techniques, we can extend this result to the case where all elements of the sequence $\a$, except for one, are bounded.

\begin{thm}\label{thm-onlyonetermbig}
Let $\a=(a_1,\dots,a_n)$ be a nearly $\Delta$-bounded sequence of positive integers with $\gcd(\a)=1$, assuming that the first $n-1$ terms of the sequence are bounded by $\Delta$. Then for any nonnegative integer $b$,
the denumerant $d(b;\a)$ can be computed using $O(n^3 \Delta^2 \log n \log \Delta)$ operations in $\Q$.
\end{thm}
\begin{proof}
Set $c = \lfloor b/a_{n}\rfloor+1$ so that the condition $0 < -b + c a_n \le a_n$ holds. Then
  \begin{align*}
  \sigma_{P(b;\a)}=&\CT_{q}\frac{q^{-b}}{(1-q^{a_{n}}y_{n})\prod_{i=1}^{n-1}(1-q^{a_{i}}y_{i})}\\
  =&\sum_{i=1}^{n-1}\CT_{q}\frac{q^{-b}}{\underline{(1-q^{a_i}y_i)}\prod\limits_{j\neq i}(1-q^{a_{j}}y_{j})}
  +\CT_{q}\frac{q^{-b}}{\underline{(1-q^{a_{n}}y_{n})}\prod\limits_{j\neq n}(1-q^{a_{j}}y_{j})}\\
  =&\sum_{i=1}^{n-1}\CT_{q}\frac{q^{-b}}{\underline{(1-q^{a_i}y_i)}\prod\limits_{j\neq i}(1-q^{a_{j}}y_{j})}
  +\sum_{i=1}^{n-1}\CT_{q}\frac{-q^{-b+ca_{n}}y_{n}^{c}}{\underline{(1-q^{a_{i}}y_{i})}\prod_{j\neq i}(1-q^{a_{j}}y_{j})}.
  \end{align*}
Here, the last equality follows from equation \eqref{equ-properandlim=0} in Proposition \ref{p-partialfraction}, whose condition is satisfied when $0 < -b + c a_n \le a_n$. Then by Definition \ref{defn-A} and Lemma \ref{lem-unit-E}, we obtain
\begin{align*}
  \sigma_{P(b;\a)}=&\sum_{i=1}^{n-1}\left(\frac{1}{a_{i}}\sum_{\zeta:\zeta^{a_{i}}=1}\left.\left(\frac{q^{-b}}{\prod_{j\neq i}(1-q^{a_{j}}y_{j})}+\frac{-q^{-b+ca_n}y_{n}^{c}}{\prod_{j\neq i}(1-q^{a_{j}}y_{j})}\right)\right|_{q=(y_{i}^{\frac{1}{a_{i}}}\zeta)^{-1}}\right)=\sum_{j} Q_{j}
\end{align*}
This is a sum of $2(a_1+\cdots+a_{n-1})\leq 2(n-1)\Delta$ rational functions $Q_{j}$ in $y_i$ involving various roots of unity $\zeta$. We then need to compute
$$ \CT_s \sigma_{P(b;\a)}\left(e^{\kappa_1s},\dots,e^{\kappa_ns}\right)=\sum_{j}\CT_s \left.Q_{j}\right|_{y_i=e^{\kappa_i s}}$$
for an admissible vector $\kappa$.

The remaining steps closely follow those in the proof of Proposition \ref{prop-polymialwhenbounded}. By Theorem \ref{theo-GTodd}, we can evaluate each term using $O(n^2 \Delta\log n\log \Delta)$ operations in $\Q$. Since there are at most $2(n-1)\Delta$ terms, this allows us to compute the denumerant $d(b;\a)$ using a total of
$$O\left(n^2 \Delta\log n\log \Delta \cdot 2(n-1)\Delta\right) = O(n^3 \Delta^2 \log n \log \Delta)$$
operations in $\Q$.

 This completes the proof.
\end{proof}

The preceding proof does not depend strongly on the variables $y_i$. In fact, the argument still holds if we replace $y_i$ with a monomial involving roots of unity, provided that we adjust the field $R_f$ accordingly. Note that this modification leads to a higher computational complexity. We summarize this observation in the following remark.

\begin{rem}\label{rem-denuwithunity}
Follow the notation as above. Let $L_{1},\dots,L_{n}$ be monomials in the variables $y_1,\dots,y_m$ with coefficient 1, and let $\zeta_{1},\dots,\zeta_{n}$ be $w$-th roots of unity. Suppose that the factors $1-q^{a_{i}}L_{i}\zeta_{i}$ are pairwise coprime.
Then, the constant term
$$\CT_{s}\left.\left(\CT_{q} \frac{1}{\underline{1-q^{a_{i}}L_{i}\zeta_{i}}} \frac{q^{-b}}{\prod_{j=1, j\neq i}^{n}(1-q^{a_{j}}L_{j}\zeta_{j})}\right)\right|_{y_{k}=e^{\kappa_{k}s},k=1,\dots,m},$$
where $\kappa$ is an admissible vector, can be computed using $O\left(n^2 wa_i^2 \log n \log (wa_i)\right)$ operations in $\Q$.
 Consequently, if $a_1,\dots, a_{n-1}$ are bounded by a constant $\Delta$, then the constant term
$$\CT_{s}\left.\left(\CT_{q}\frac{q^{-b}}{\prod_{i=1}^{n}(1-q^{a_{i}}L_{i}\zeta_{i})}\right)\right|_{y_{k}=e^{\kappa_{k}s},k=1,\dots,m}$$
can be computed using $O\left(n^3 w\Delta^2 \log n \log (w\Delta)\right)$ operations in $\Q$.
\end{rem}

\section{Arithmetic complexity for \texorpdfstring{$B$}{B}-bounded A-L denumerants}\label{sec-Bbounded}
In this section, we present an arithmetic algorithm for $B$-bounded A-L denumerants. Recall that for a given $n \in \mathbb{Z}_{+}$, a denumerant $d(b;\a)$ is called a $B$-bounded A-L denumerant if $\a$ can be written as $\a = \mathbf{p}M + \mathbf{r}N$. Here, $\mathbf{p}=(p_{1},\dots,p_{n})$ and $\mathbf{r}=(r_1,\dots,r_{n})$ are linearly independent integer sequences, $M$ and $N$ are coprime positive integers, and $|p_{i}|, |r_{i}| \leq B$ for all $i$. In what follows, we define $D_{i,j} := r_i p_j - r_j p_i$.

Since $\a$ is not necessarily bounded, existing results do not directly apply to this class of denumerants. However, we show that the computation of $d(b;\a)$ can be reduced to a set of subproblems. Each subproblem is then solved using Theorem \ref{thm-onlyonetermbig} (specifically, Remark \ref{rem-denuwithunity}).

Before proceeding, we introduce a running example for clarity.

\begin{exa}
Consider the sequence $$\a = (a_1, \dots, a_6) = (12228, 36679, 36682, 48908, 61139, 73365).$$ In this example, we take $M=12228$, $N=1$, $\mathbf{p} = (p_1, \dots, p_6) = (1, 3, 3, 4, 5, 6)$, and $\mathbf{r} = (r_1, \dots, r_6) = (0, -5, -2, -4, -1, -3)$. Our goal is to evaluate $d(b;\a)$ at $b=897168380$.
\end{exa}
This is one of the instances generated by Cornuéjols et al.\ in \cite{CUWW97} and later revisited by Aardal and Lenstra in \cite{AL04}. In what follows, we append the symbol ``$\to$'' to certain formulas to indicate their exact values in this specific example.

As a starting point, our approach relies on the following result.
\begin{prop}[Xin, Y.~Zhang, and Z.~Zhang~{\cite[Proposition~3]{xin2024combinatorial}}]\label{prop-keytrans}
If $\gcd(k,a)=1$, then for any rational function $F(x)$ whose denominator is relatively prime to $1-ux^{a}$, we have
 \begin{align*}
   \CT_{x}\frac{1}{\underline{1-ux^{a}}}F(x)=\CT_{x}\frac{1}{\underline{1-u^{1/k}x^{a}}}F(x^{k})
 \end{align*}
\end{prop}

First, we have $d(b;\a)=\sigma_{P(b;\a)}(\mathbf{1})$ with
\begin{align*}
\sigma_{P(b;\a)}&=\CT_{q}\frac{q^{-b}}{(1-q^{a_{1}}y_{1})\cdots(1-q^{a_{n}}y_{n})} \\
 &\to\CT_{q}\frac{q^{-897168380}}{(1-q^{12228}y_{1})\cdots(1-q^{73365}y_{6})}.
\end{align*}
Clearly, these numbers are too large to be handled directly.

We choose the field of iterated Laurent series, defined by the variable order $\texttt{vars}=[t,q,y_{1},\dots,y_{n},y_{n+1}]$ in \eqref{Ite-Laurent}, to be our working field. Here, $y_{n+1}$ is an auxiliary variable that we add later. In this field, for any $c\in\N$, we have
$$\CT_{t}\frac{t^{-c}}{1-t/q^{M}}=\CT_{t}t^{-c}\sum_{i\geq 0}(t/q^{M})^{i}=q^{-Mc},$$ then
\begin{align*}
 \sigma_{P(b;\a)}&=\CT_{q,t}\frac{q^{-b+Mc}\cdot t^{-c}}{(1-q^{a_{1}}y_{1})\cdots(1-q^{a_{n}}y_{n})(1-t/q^{M})}\\
  &=\CT_{q,t}\frac{-q^{-b+Mc+M}\cdot t^{-c-1}}{(1-q^{a_{1}}y_{1})\cdots(1-q^{a_{n}}y_{n})(1-q^{M}/t)}\\
  &\to\CT_{q,t}\frac{-q^{-897168380+12228(c+1)}t^{-c-1}}{(1-q^{12228}y_{1})\cdots(1-q^{73365}y_{6})(1-q^{12228}/t)}.
\end{align*}

Let us set $c = \lfloor b/M \rfloor \geq 0$, which ensures that $0 < -b+Mc+M \leq M$. Since $M$ and $N$ are coprime, we can choose integers $u,v \in \Z$ such that $uM+vN = -b+Mc+M$ and $0 < v - \sum_{i: r_i < 0} r_i \leq M$. Consequently, we have
\begin{align*}
&\CT_{q}\frac{-q^{-b+Mc+M}\cdot t^{-c-1}}{(1-q^{a_{1}}y_1)\cdots(1-q^{a_{n}}y_n)(1-q^{M}/t)}\\
(\text{by Proposition \ref{p-partialfraction}})=&\CT_{q}\frac{q^{-b+Mc+M}\cdot t^{-c-1}}{(1-q^{a_{1}}y_1)\cdots(1-q^{a_{n}}y_n)\underline{(1-q^{M}/t)}}\\
=&\CT_{q}\frac{q^{vN}\cdot t^{-c-1+u}}{(1-q^{r_{1}N}t^{p_{1}}y_1)\cdots(1-q^{r_{n}N}t^{p_{n}}y_n)\underline{(1-q^{M}/t)}}\\
(\text{by Proposition \ref{prop-keytrans}})=&\CT_{q}\frac{q^{v}\cdot t^{-c-1+u}}{(1-q^{r_{1}}t^{p_{1}}y_1)\cdots(1-q^{r_{n}}t^{p_{n}}y_n)\underline{(1-q^{M}/t^{N})}}   \\
(\text{by Proposition \ref{p-partialfraction}})=&\CT_{q}\frac{-q^{v}\cdot t^{-c-1+u}}{\underline{(1-q^{r_{1}}t^{p_{1}}y_{1})\cdots(1-q^{r_{n}}t^{p_{n}}y_{n})}(1-q^{M}/t^{N})}\\
\to&\CT_{q}\frac{-q^{12208}t^{-73371}}{\underline{(1-ty_{1})(1-q^{-5}t^{3}y_{2})\cdots(1-q^{-3}t^{6}y_{6})}(1-q^{12228}/t)},
\end{align*}
where we also rely on the condition that the $r_i$ are not all zero, a property guaranteed by the linear independence of $(p_{1}, p_{2}, \dots, p_{n})$ and $(r_{1}, r_{2}, \dots, r_{n})$.

After all, we have the following proposition.
\begin{prop}\label{prop-LAprobqtformula}
Given $n\in\Z_{+}$, let $\a = (a_1, \dots, a_n)$ be defined by $a_i = p_i M + r_i N \in \Z_{+}$. Here, $M, N \in \Z_{+}$ are coprime, $p_i, r_i \in \Z$, and the vectors $(p_1, \dots, p_n)$ and $(r_1, \dots, r_n)$ are linearly independent. By choosing integers $u,v \in\Z$ such that $uM+vN = -b+M\lfloor b/M \rfloor+M$ and $0 < v-\sum_{i:r_i<0}r_i \leq M$, we have
 \begin{align*}
 \CT_{q}\frac{q^{-b}}{(1-q^{a_{1}}y_{1})\cdots(1-q^{a_{n}}y_{n})}&=-\CT_{t}\left(\CT_{q}\frac{q^{v}\cdot t^{-\lfloor\frac{b}{M}\rfloor-1+u}}{\underline{(1-q^{r_{1}}t^{p_{1}}y_{1})\cdots(1-q^{r_{n}}t^{p_{n}}y_{n})}(1-q^{M}/t^{N})}\right)\notag\\
 &=-\sum_{i=1}^{n}\CT_{t}\left(\CT_{q}\frac{q^{v}\cdot t^{-\lfloor\frac{b}{M}\rfloor-1+u}}{\underline{(1-q^{r_{i}}t^{p_{i}}y_{i})}\prod\limits_{j\neq i}(1-q^{r_{j}}t^{p_{j}}y_{j})(1-q^{M}/t^{N})}\right).
\end{align*}
\end{prop}

By introducing the auxiliary variable $t$ and reformulating the generating function, we successfully isolate the extreme parameters $M$ and $N$. This allows us to achieve our goal: writing $d(b;\a)$ as a sum of weighted $\Delta$-bounded denumerants.

\begin{thm}\label{thm-crudealgocomplexity}
Let $n, B \in \Z_{+}$, and let $\a$ be a $B$-bounded A-L sequence. This means there exist coprime positive integers $M$ and $N$ such that $a_i = p_i M + r_i N$, where the integer vectors $\mathbf{p} = (p_1, \dots, p_n)$ and $\mathbf{r} = (r_1, \dots, r_n)$ are linearly independent, and $|p_i|, |r_i| \le B$ for all $i$. Given this representation and any nonnegative integer $b$, we can compute $d(b;\a)$ using
$$O(\log(2+\min\{M,N\}))+O(n^4 B^6 \log n \log B)$$
arithmetic operations.
\end{thm}
\begin{proof}
First, by Proposition \ref{prop-LAprobqtformula}, we have
\begin{align}
\sigma_{P(b;\a)}(\y)=&-\sum_{i=1}^{n}\CT_{t}\left(\CT_{q}\frac{q^{v}\cdot t^{-\lfloor\frac{b}{M}\rfloor-1+u}}{\underline{(1-q^{r_{i}}t^{p_{i}}y_{i})}\prod\limits_{j\neq i}(1-q^{r_{j}}t^{p_{j}}y_{j})(1-q^{M}/t^{N})}\right),\label{equ-LAprobqtformula}
\end{align}
and $d(b;\a)=\lim\limits_{\y\to\1}\sigma_{P(b;\a)}$.

For a given $i \in \{1,\dots,n\}$, Definition \ref{defn-A} implies that the corresponding term in \eqref{equ-LAprobqtformula} vanishes when $r_i = 0$. If $r_i \neq 0$, we can assume without loss of generality that $r_i > 0$, since the case $r_i < 0$ can be handled similarly. Then, by Lemma \ref{lem-unit-E}, we obtain
\begin{align}
&\CT_{t}\left(\CT_{q}\frac{q^{v}\cdot t^{-\lfloor\frac{b}{M}\rfloor-1+u}}{\underline{(1-q^{r_{i}}t^{p_{i}}y_{i})}\prod\limits_{j\neq i}(1-q^{r_{j}}t^{p_{j}}y_{j})(1-q^{M}/t^{N})}\right)\notag\\
=&\frac{1}{r_{i}}\sum_{\zeta:\zeta^{r_{i}}=1}\CT_{t}\frac{ t^{-\frac{vp_{i}}{r_{i}}-\lfloor\frac{b}{M}\rfloor-1+u}y_{i}^{-\frac{v}{r_{i}}}\zeta^{-v}}{\prod\limits_{j:j\neq i}(1-t^{\frac{D_{i,j}}{r_{i}}}y_{i}^{-\frac{r_{j}}{r_{i}}}y_{j}\zeta^{-r_{j}})(1-t^{-\frac{a_{i}}{r_{i}}}y_{i}^{-\frac{M}{r_{i}}}\zeta^{-M})}\notag\\
=&\frac{1}{r_{i}}\sum_{\zeta:\zeta^{r_{i}}=1}\CT_{t}\frac{ t^{-vp_{i}-r_{i}\lfloor\frac{b}{M}\rfloor-r_{i}+r_{i}u}y_{i}^{-\frac{v}{r_i}}\zeta^{-v}}{\prod\limits_{j:j\neq i}(1-t^{D_{i,j}}y_{i}^{-\frac{r_{j}}{r_{i}}}y_{j}\zeta^{-r_{j}})(1-t^{-a_{i}}y_{i}^{-\frac{M}{r_{i}}}\zeta^{-M})},\label{equ-onlytwithoutq}
\end{align}
where the last step follows from the substitution $t \mapsto t^{r_i}$, which preserves the constant term since $r_i > 0$. Therefore, when $r_i < 0$, the only difference is that we instead use the substitution $t \mapsto t^{-r_i}$.

Note that if
$$-vp_{i}-r_{i}\lfloor\frac{b}{M}\rfloor-r_{i}+r_{i}u-\sum_{j:D_{i,j}<0}D_{i,j}+a_{i}>0,$$
it immediately follows that the value of \eqref{equ-onlytwithoutq} is 0.
Therefore, in what follows, we can assume that
\begin{equation}\label{inequ-<0}
-vp_{i}-r_{i}\lfloor\frac{b}{M}\rfloor-r_{i}+r_{i}u-\sum_{j:D_{i,j}<0}D_{i,j}+a_{i}\le 0.
\end{equation}
Since $|r_i|, |p_i| \leq B$, each term in the sum is a weighted nearly $2B^2$-bounded denumerant, meaning Theorem \ref{thm-onlyonetermbig} cannot be applied directly.
Instead, we must follow the steps in the proof of Theorem \ref{thm-onlyonetermbig} to write $\sigma_{P(b;\a)}$ as a large sum of simple rational functions in the variables $y_i$.
We then choose a global admissible vector $\kappa$ to perform the substitution $y_i \mapsto e^{\kappa_i s}$.
This substitution is applied to \eqref{equ-onlytwithoutq}, where $\zeta$ is a $|r_i|$-th root of unity and $|r_i| \leq B$.
Together with \eqref{inequ-<0}, this ensures that Remark \ref{rem-denuwithunity} is applicable.
By Remark \ref{rem-denuwithunity}, each term takes $O(n^3 B^5 \log n \log B)$ operations in $\Q$.
Since evaluating $d(b;\a)$ requires computing at most $nB$ such terms, the total computation takes $O(n^4 B^6 \log n \log B)$ time. The extended Euclidean algorithm on $M,N$ takes $O(\log(2+\min\{M,N\}))$ operations to obtain $u,v$. 

This completes the proof.
\end{proof}

While the proposed algorithm is effective, it still exhibits considerable computational redundancy. In the next section, we leverage additional constant term techniques to develop a refined algorithm for $\Delta$-modular A-L denumerants.

\section{Arithmetic complexity for \texorpdfstring{$\Delta$}{Delta}-modular A-L denumerants}\label{sec-5}
During the derivation in Section \ref{sec-Bbounded}, the iterative application of Lemma \ref{lem-unit-E} produces groups of similar terms involving roots of unity. Xin, Xu, and Zhang~\cite{XXZ25} provide a solution to this issue. We first outline their approach and then apply it to optimize our algorithm.

\subsection{Constant Term Extraction Over Matrix Form}

We recall the ingredients of their approach that are needed below.

\begin{dfn}[Xin, Xu, and Zhang~{\cite[Definition~4]{XXZ25}}]\label{defn-Z_y}
Define the operators $\mathcal{Z}_{y_j}$ acting on fractional series by
\begin{equation}
  \mathcal{Z}_{y_j}\sum_{\kappa\in\Q^n}a_{\kappa}\y^{\kappa}:=\sum_{\kappa\in\Q^{n},\kappa_j\in \Z}a_{\kappa}\y^{\kappa},
\end{equation}
where the coefficients $a_{\kappa}$ are free of the $y$ variables.
In other words, $\mathcal{Z}_{y_j}$ extracts all terms with integer exponents of $y_j$.
\end{dfn}

We will consider a rational function $E$ of \emph{the standard form} in $\lambda$
\begin{align}\label{e-integer-form}
E= \frac{L(\lambda)}{\prod_{j=1}^n (1-u_j \lambda^{b_j})},
\end{align}
where $b_j \in \Q$ and the denominator factors are coprime to each other.

\begin{dfn}[Xin, Xu, and Zhang~{\cite[Definition~5]{XXZ25}}]\label{defn-CT-Z}
Suppose $E$ is given as in \eqref{e-integer-form}.
If $b_j\in \Q$ is nonzero and the variable $y$ only appears at $u_j=yu_j'$, then define
$$ \CT_{\lambda, j} E:=\sgn(b_j) \mathcal{Z}_{y} \left.\left(E(1-u_j \lambda^{b_j}) \right)\right|_{\lambda= (yu_j')^{-1/{b_j}}}.$$
\end{dfn}

The following proposition and corollary characterize the properties of this operator.

\begin{prop}[Xin, Xu, and Zhang~{\cite[Proposition~7]{XXZ25}}]\label{prop-singleCT}
Suppose $E$ is defined as in \eqref{e-integer-form}, where the variable $y$ appears exclusively in the relation $u_j = y u_j'$. If the substituted expression $\left. E \right|_{\lambda \mapsto \lambda^m}$ is integral in $\lambda$ for a positive integer $m$, then its constant term is independent of $m$; that is,
\begin{equation}
\CT_{\lambda, j} \left( \left.E \right|_{\lambda\mapsto \lambda^m} \right) = \CT_{\lambda, j} E.
\end{equation}
Moreover, if $E$ itself is integral in $\lambda$, we obtain
\begin{equation}
\CT_\lambda \frac{1}{\underline{1-u_j \lambda^{b_j}}} E(1-u_j \lambda^{b_j}) = \CT_{\lambda, j} E.
\end{equation}
\end{prop}

\begin{cor}[adapted from Xin, Xu, and Zhang~{\cite[Corollary~8 and Lemma~17]{XXZ25}}]\label{cor-contri-dual}
Suppose $E$ is a fractional rational function given as in \eqref{e-integer-form} where $y_j$ appears only at $u_j=y_ju_j'$.
If $E$ is proper in $\lambda$, i.e., $\deg_\lambda E < 0$, then
\begin{equation}\label{equ-matrixproper}
  \CT_\lambda E = \sum_{(u_j \lambda^{b_j})^{\sgn(b_j)} <1}\CT_{\lambda, j} E.
\end{equation}
Moreover, if $\deg_\lambda E < 0$ and $E|_{\lambda=0}=0$, we deduce that
\begin{equation}\label{equ-matrixproperandlim=0}
\CT_{\lambda, 1} E+\CT_{\lambda, 2} E+\cdots+\CT_{\lambda, n} E=0.
\end{equation}
\end{cor}

The following discussion illustrates how constant term operations can be linked to elementary matrix operations. Let $P(A,\b)=\{\alpha\in\R_{\geq0}^{n}: A\alpha=\b\}$, where $A$ is an $r\times n$ integer matrix and $\b$ is an integral vector in $\Z^r$.
Its lattice point generating function is denoted by
$$\sigma_{P(A,\b)}=\sum_{\alpha\in P(A,\b)\cap \Z^{n}}\y^{\alpha}=\sum_{\alpha\in P(A,\b)\cap \Z^{n}}y_{1}^{\alpha_{1}}\cdots y_{n}^{\alpha_{n}}.$$
 There is a well-known constant term formula:
\begin{align}\label{equ-inhomo}
  \sigma_{P(A,\b)}=\CT_{\Lambda} \frac{\Lambda^{-\b}}{
\prod_{j=1}^n (1- \Lambda^{A_j} y_j)}, \text{where $A_j$ is the $j$-th column  of $A$}.
\end{align}
The above rational function is encoded by the following matrix:
\begin{equation}\label{equ-M}
M=\left(\begin{array}{cc}
         \textit{id}_n& \0\\
         A            &-\b
        \end{array}\right)=(m_{i,j})_{(n+r)\times{(n+1)}},
\end{equation}
where $\textit{id}_n$ denotes the identity matrix of order $n$.

The matrix $M \leftarrow (i,j)$ is defined as the result of applying Gaussian column elimination to $M$ using the nonzero entry $m_{n+i, j}$ as the \emph{pivot}. Specifically, $-\frac{m_{n+i,s}}{m_{n+i,j}}$ times the $j$-th column is added to the $s$-th column for all $s \neq j$, after which row $n+i$ and column $j$ are omitted from subsequent steps.
Recursively, this operation extends to$$M \leftarrow \left((i_1,\dots, i_k);(j_1,\dots, j_k)\right) = M \leftarrow (i_1,j_1) \leftarrow \cdots \leftarrow (i_k,j_k).$$
Letting $(p_1, \dots, p_k)$ denote the sequence of pivots—where $p_\ell$ is the $(n+i_{\ell}, j_\ell)$-th entry of $M \leftarrow ((i_1,\dots, i_{\ell-1});(j_1,\dots, j_{\ell-1}))$—the product $p_1 \cdots p_k$ equals the $k \times k$ minor $A\begin{pmatrix} i_1,\dots,i_k \\ j_1,\dots,j_k \end{pmatrix}$.
This corresponds exactly to the determinant of the submatrix of $A$ formed by rows $i_s$ and columns $j_t$.

\begin{prop}[Xin, Xu, and Zhang~{\cite[Proposition~10]{XXZ25}}]
Suppose $M$ is given in \eqref{equ-M}. Then for any
permutations $\pi$ and $\pi'$ on $1,2,\dots, k$, we have
 $$M\leftarrow ((i_1,\dots, i_k);(j_1,\dots, j_k))= M\leftarrow ((i_{\pi(1)},\dots, i_{\pi(k)});(j_{\pi'(1)},\dots, j_{\pi'(k)})),$$
if both sides are well-defined and understood as ignored rows and columns removed.
\end{prop}
The above results demonstrate that the elimination result is independent of the chosen order. For convenience, we denote the matrix $M\leftarrow ((i_1,\dots, i_k);(j_1,\dots, j_k))$ by $M_k\langle I;J\rangle$, where $I=\{i_1,\dots, i_k\},J=\{j_1,\dots,j_k\}$.

We introduce the notation
\begin{equation}\label{equ-Psi}
[M_k\langle I;J\rangle]:=\frac{(\y,\Lambda)^{\gamma_{n+1}}}{\left(1-(\y,\Lambda)^{\gamma_1}\right)\cdots \left(1-(\y,\Lambda)^{\gamma_n}\right)},
\end{equation}
where $\gamma_j \in \Q^{n+r}$ is the $j$-th column of $M_k\langle I;J\rangle$. For any $j \in J$, the corresponding denominator factor is understood to be $1$. Additionally, $M_r\langle J\rangle$ serves as shorthand for $M_r\langle [r];J\rangle$.

The operators $\CT\limits_{\lambda_i,j}$ and $\mathcal{Z}_{\y}$ are commutative.
 \begin{prop}[Xin, Xu, and Zhang~{\cite[Proposition~11]{XXZ25}}]\label{prop-commu}
 Using the preceding notation, if
$i \not\in I,~j\not\in J$, then we have
 \begin{equation}
\CT_{\lambda_i,j}\mathcal{Z}_{\y} [M_k\langle I;J\rangle]=\mathcal{Z}_{\y}\CT_{\lambda_i,j}[M_k\langle I;J\rangle].
\end{equation}
\end{prop}
Finally, the commutativity of these operators leads to our key lemma.
\begin{lem}[Xin, Xu, and Zhang~{\cite[Lemma~13]{XXZ25}}]\label{lem-CT-cone}
Suppose $M$ is given in \eqref{equ-M}. Then for $k\leq r$, we have
$$\CT\limits_{\lambda_{i_k},j_k}\cdots\CT\limits_{\lambda_{i_2},j_2}\CT\limits_{\lambda_{i_1},j_1} [M]=\sgn\left(A\begin{pmatrix}
i_1,\dots,i_{k}\\
j_1,\dots,j_{k}
\end{pmatrix}\right)\mathcal{Z}_{\y}[M_k\langle \{i_1,\dots,i_{k}\};\{j_1,\dots,j_{k}\}\rangle],$$
provided that both sides are well-defined.
\end{lem}

\subsection{Refined Algorithm}\label{subsec-refinedalg}
Following the notation from Section \ref{sec-Bbounded}, we introduce an auxiliary variable $y_{n+1}$ and define the new function
\begin{align}
 \bar{\sigma}_{P(b;\a)}=&-\CT_{t}\left(\CT_{q}\frac{q^{v}\cdot t^{-\lfloor\frac{b}{M}\rfloor-1+u}}{\underline{(1-q^{r_{1}}t^{p_{1}}y_{1})\cdots(1-q^{r_{n}}t^{p_{n}}y_{n})}(1-y_{n+1}q^{M}/t^{N})}\right),\label{barF}
\end{align}
where $u$ and $v$ are chosen as in Proposition \ref{prop-LAprobqtformula}. Note that
\begin{align}
\left.\bar{\sigma}_{P(b;\a)}\right|_{y_{n+1}=1}=\sigma_{P(b;\a)} \quad\text{and}\quad d(b;\a)=\bar{\sigma}_{P(b;\a)}(\mathbf{1}).\label{dval}
\end{align}
Furthermore, we recall that our working field is the field of iterated Laurent series with the variable ordering $\texttt{vars}=[t,q,y_1,\dots,y_n,y_{n+1}]$ given in \eqref{Ite-Laurent}.

\begin{thm}\label{Thm-Center}
Let $n \in \Z_{+}$ and consider a sequence of positive integers $\a = (a_1, \dots, a_n)$, where $a_i = p_i M + r_i N$.
Assume $M, N \in \Z_{+}$ are coprime, $p_i, r_i \in \Z$, and the vectors $(p_1, \dots, p_n)$ and $(r_1, \dots, r_n)$ are linearly independent.
For a given $b \in \N$, we fix integers $u,v$ satisfying $uM + vN =-b+M\lfloor b/M\rfloor+M$ and $0 < v - \sum_{i: r_i < 0} r_i \leq M$.
We define the block matrix $A$ as
\[A:=
\left(
\begin{array}{cccc|c}
1 & \cdots & 0 & 0 & 0 \\
\vdots & \ddots & \vdots & \vdots & \vdots \\
0 & \cdots & 1 & 0 & 0\\
0 & \cdots & 0 & 1 & 0\\
\hline
r_{1} & \cdots & r_{n} & M & v\\
p_{1} & \cdots & p_{n} & -N & -\lfloor\frac{b}{M}\rfloor-1+u
\end{array}
\right),
\]
Next, define the index sets $J := [n] \setminus I$ and
$$I:=\left\{i:r_{i}\neq0, -\frac{vp_{i}}{r_{i}}-\lfloor\frac{b}{M}\rfloor-1+u-\sum\limits_{1\leq j\leq n}\chi(\frac{D_{i,j}}{r_{i}}<0)\frac{D_{i,j}}{r_{i}}+\chi(r_{i}>0)\frac{a_{i}}
{r_{i}}\leq0\right\}$$
where $\chi$ is the standard indicator function.
For each $i \in I$, let $A^{(i)}$ be the matrix obtained by subtracting $s_i$ times the $(n+1)$-th column of $A$ from its $(n+2)$-th column, where
$$s_i =
\begin{cases}
\left\lfloor \frac{1}{a_i}\left(vp_i + \left(\lfloor \frac{b}{M} \rfloor + 1 - u\right)r_i + \sum_{j: D_{i,j}<0}D_{i,j}\right) \right\rfloor, & \text{if } r_i > 0,\\[6pt]
\left\lceil \frac{1}{a_i}\left(vp_i + \left(\lfloor \frac{b}{M} \rfloor + 1 - u\right)r_i + \sum_{j: D_{i,j}>0}D_{i,j}\right) \right\rceil - 1, & \text{if } r_i < 0.
\end{cases}$$
Finally, denoting $\langle i,j \rangle := \CT\limits_{t,j}\CT\limits_{q,i}[A]$ and $\langle i,j \rangle^{\prime} := \CT\limits_{t,j}\CT\limits_{q,i}[A^{(i)}]$, we obtain the identity
\begin{equation}\label{equa-sumij}
\bar{\sigma}_{P(b;\a)} = -\sum_{i\in I, j\in J}\langle i,j\rangle + \sum_{i\in I, 1\leq j\leq n}\langle i,j\rangle^{\prime}.
\end{equation}
\end{thm}
\begin{proof}

Recalling the definition in \eqref{equ-Psi}, we have
$$[A]=\frac{q^{v}\cdot t^{-\lfloor\frac{b}{M}\rfloor-1+u}}{(1-q^{r_{1}}t^{p_{1}}y_{1})\cdots(1-q^{r_{n}}t^{p_{n}}y_{n})(1-y_{n+1}q^{M}/t^{N})}.$$
Invoking equation \eqref{barF}, Proposition \ref{prop-singleCT}, and Lemma \ref{lem-CT-cone} yields
\begin{align*}
  \bar{\sigma}_{P(b;\a)}=&-\CT_{t}\left(\sum_{i:r_{i}\neq0}\CT_{q}\frac{q^{v}\cdot t^{-\lfloor\frac{b}{M}\rfloor-1+u}}{\underline{(1-q^{r_{i}}t^{p_{i}}y_{i})}\cdot\prod\limits_{j\neq i}(1-q^{r_{j}}t^{p_{j}}y_{j})(1-y_{n+1}q^{M}/t^{N})}\right)\\
  =&-\sum_{i:r_{i}\neq0}\CT_{t}\CT_{q,i}[A]=-\sum_{i:r_{i}\neq0}\CT_{t} \sgn(r_{i})\mathcal{Z}_{\y}[A_1\langle \{1\};\{i\}\rangle]\\
  =&-\sum_{i:r_{i}\neq0}\sgn(r_{i})\mathcal{Z}_{\y}\CT_{t}[A_1\langle \{1\};\{i\}\rangle],
\end{align*}
where
\begin{align*}
  [A_1\langle \{1\};\{i\}\rangle]&=\frac{ t^{-\frac{vp_{i}}{r_{i}}-\lfloor\frac{b}{M}\rfloor-1+u}y_{i}^{-v/r_{i}}}{\prod_{j=1}^{i-1}(1-t^{\frac{D_{i,j}}{r_{i}}}y_{j}y_{i}^{-\frac{r_{j}}{r_{i}}})(1)\prod_{j=i+1}^{n}(1-t^{\frac{D_{i,j}}{r_{i}}}y_{j}y_{i}^{-\frac{r_{j}}{r_{i}}})(1-t^{-\frac{a_{i}}{r_{i}}}y_{n+1}y_{i}^{-\frac{M}{r_{i}}})}.
\end{align*}
Observe that if
$$-\frac{vp_{i}}{r_{i}}-\lfloor\frac{b}{M}\rfloor-1+u-\sum\limits_{1\leq j\leq n}\chi(\frac{D_{i,j}}{r_{i}}<0)\frac{D_{i,j}}{r_{i}}+\chi(r_{i}>0)\frac{a_{i}}{r_{i}}>0,$$ it immediately follows that $\CT\limits_{t}[A_1\langle \{1\};\{i\}\rangle]=0$. Consequently, the summation reduces to
$$\bar{\sigma}_{P(b;\a)}=-\sum_{i\in I}\sgn(r_{i})\mathcal{Z}_{\y}\CT_{t}[A_1\langle \{1\};\{i\}\rangle].$$
For every $i\in I$, it is clear that $[A_1\langle \{1\};\{i\}\rangle]$ is proper with respect to $t$. In this case, applying equation \eqref{equ-matrixproper} from Corollary \ref{cor-contri-dual} and Proposition \ref{prop-commu} gives
\begin{align*}
 \sgn(r_{i})\mathcal{Z}_{\y}\CT_{t}[A_1\langle \{1\};\{i\}\rangle]&=\sum_{1\leq j\leq n+1}\sgn(r_{i})\mathcal{Z}_{\y}\CT_{t,j}[A_1\langle \{1\};\{i\}\rangle]\\
 &=\sum_{1\leq j\leq n+1}\CT_{t,j}\sgn(r_{i})\mathcal{Z}_{\y}[A_1\langle \{1\};\{i\}\rangle]\\
 &=\sum_{1\leq j\leq n+1}\CT\limits_{t,j}\CT\limits_{q,i}[A]=\sum_{1\leq j\leq n+1}\langle i,j\rangle,
\end{align*}
Here, $\langle i, j \rangle$ is defined to be $0$ whenever $D_{i,j} = 0$.
  Thus, we obtain
  \begin{align}
  \bar{\sigma}_{P(b;\a)}=&-\sum_{i\in I, 1\leq j\leq n+1}\langle i,j\rangle\notag\\
  =&-\sum_{i\in I,1\leq j\leq n}\langle i,j\rangle-\sum_{i\in I}\langle i,n+1\rangle\notag\\
  =&-\sum_{i\in I,j\in J}\langle i,j\rangle-\sum_{i\in I}\langle i,n+1\rangle,\label{ijoffset}
  \end{align}
  where the last equality follows from the fact that $\langle i,j\rangle+\langle j,i\rangle=0$ by Lemma \ref{lem-CT-cone}.

 Clearly, $\langle i,n+1\rangle = \langle i,n+1\rangle^{\prime}$ for every $i \in I$; thus, it follows that
  \begin{align*}
  \langle i,n+1\rangle=\langle i,n+1\rangle^{\prime}=\CT\limits_{t,n+1}\CT\limits_{q,i}[A^{(i)}]\notag=\sgn(r_{i})\mathcal{Z}_{\y}\CT\limits_{t,n+1}[A^{(i)}_1\langle \{1\};\{i\}\rangle],
  \end{align*}
  where the rational function
  \begin{align*}
   [A^{(i)}_1\langle \{1\};\{i\}\rangle] =&\frac{ t^{-\frac{vp_{i}}{r_{i}}-\lfloor\frac{b}{M}\rfloor-1+u+s_{i}\frac{a_{i}}{r_{i}}}y_{i}^{(-v+Ms_{i})/r_{i}}y_{n+1}^{-s_{i}}}{\prod_{j=1}^{i-1}(1-t^{\frac{D_{i,j}}{r_{i}}}y_{j}y_{i}^{-\frac{r_{j}}{r_{i}}})(1)\prod_{j=i+1}^{n}(1-t^{\frac{D_{i,j}}{r_{i}}}y_{j}y_{i}^{-\frac{r_{j}}{r_{i}}})(1-t^{-\frac{a_{i}}{r_{i}}}y_{n+1}y_{i}^{-\frac{M}{r_{i}}})}
  \end{align*}
is proper with respect to $t$ and evaluates to $0$ at $t=0$ for this special value of $s_{i}$. Hence, equation \eqref{equ-matrixproperandlim=0} in Corollary \ref{cor-contri-dual} implies
 \begin{align*}
   \langle i,n+1\rangle=\sgn(r_{i})\mathcal{Z}_{\y}\CT\limits_{t,n+1}[A^{(i)}_1\langle \{1\};\{i\}\rangle]=-\sum_{j=1}^{n}\sgn(r_{i})\mathcal{Z}_{\y}\CT\limits_{t,j}[A^{(i)}_1\langle \{1\};\{i\}\rangle]=-\sum_{j=1}^{n}\langle i,j\rangle^{\prime}.
 \end{align*}

 Finally, substituting this back into equation \eqref{ijoffset} yields
  $$\bar{\sigma}_{P(b;\a)}=-\sum_{i\in I,j\in J}\langle i,j\rangle+\sum_{i\in I,1\leq j\leq n}\langle i,j\rangle^{\prime}.$$

This completes the proof.
\end{proof}

We now evaluate an individual term $\langle i,j\rangle$ or $\langle i,j\rangle^{\prime}$ in \eqref{equa-sumij}. If $D_{i,j} = 0$, then both $\langle i,j\rangle$ and $\langle i,j\rangle^{\prime}$ are defined to be $0$. Thus, we may assume $D_{i,j} \neq 0$ in what follows.

\begin{prop}
 Each term in \eqref{equa-sumij}, whether $\langle i,j\rangle$ or $\langle i,j\rangle^{\prime}$, can be expressed as a sum of $|D_{i,j}|$ terms. Each of these summands takes the form
\begin{equation}\label{equ-finalformula}
\frac{1}{D_{i,j}}\frac{L_{0}\xi_{0}}{\prod_{k=1}^{n-1}(1-L_{k}\xi_{k})},
\end{equation}
where $\xi_0,\dots,\xi_{n-1}$ are $|D_{i,j}|$-th roots of unity, and $L_0,\dots,L_{n-1}$ are monomials in $y_1,\dots,y_{n+1}$ with coefficient 1.
\end{prop}
\begin{proof}
Without loss of generality, we assume $i=1$ and $j=2$ by permuting the denominator factors and renaming the $y$-variables. We now outline the procedure for computing $\langle 1,2 \rangle$. The computation for $\langle 1,2 \rangle^{\prime}$ follows a similar process.

Let $S_1=\left(
\begin{array}{cc}
r_{1} & r_{2}\\
p_{1} & p_{2}
\end{array}
\right)$,
$S_2=\left(
\begin{array}{cccc}
r_{3} & \dots & r_{n} & M\\
p_{3} & \dots & p_{n} & -N
\end{array}
\right)$,
$A=\left(
\begin{array}{ccc}
\textit{id}_{2} & \mathbf{0} & \mathbf{0}\\
\mathbf{0} & \textit{id}_{n-1} & \mathbf{0}\\
S_{1} & S_{2} & \b
\end{array}
\right)$. By Lemma \ref{lem-CT-cone}, we have
\begin{align}
\langle 1,2\rangle=\CT\limits_{t,2}\CT\limits_{q,1}[A]=\sgn(\det(S_{1}))\mathcal{Z}_{\y}[A_{2}\langle[2]
\rangle],\label{equ-unibefore}
\end{align}
where $A_{2}\langle[2]\rangle=\begin{pmatrix}
-{S_1}^{-1}S_2& -{S_1}^{-1}\b\\
\textit{id}_{n-1}&\0
\end{pmatrix}$.

We compute the Smith normal form of the submatrix $S_1$ as $H = US_1V$, where $U$ and $V$ are unimodular matrices, and $H$ is a diagonal matrix with diagonal entries $h_1$ and $h_2$ such that $h_1 \mid h_2$. Note that $h_1 h_2 = \vert{}\det(S_1)\vert{} = \vert{}r_1 p_2 - r_2 p_1\vert{}$.

Now we construct the matrix
\begin{align}\label{eq_N}
 \hat{A}=\begin{pmatrix}
         \textit{id}_{2} & \0 & \0\\
         \0 & \textit{id}_{n-1} & \0 \\
         H & U S_2  &U\b
      \end{pmatrix}
\end{align}
such that
\begin{align}\label{e-M-M^}
    \hat{A}_{2}\langle [2]\rangle = \begin{pmatrix}
-V^{-1}{S_1}^{-1}S_2&-V^{-1}{S_1}^{-1}\b\\
\textit{id}_{n-1}&\0
\end{pmatrix} = \text{diag}(V^{-1},\textit{id}_{n-1})  A_{2}\langle [2]\rangle.
\end{align}
Define the action of a nonsingular matrix \( W = (w_{ij})_{n \times n} \) by
$$ W \circ F(\y) = F(\y^{W e_1}, \y^{W e_2}, \dots, \y^{W e_n}), \quad \text{where} \quad  \y^{W e_i} = y_1^{w_{1i}} \cdots y_n^{w_{ni}}. $$
Then we have
\begin{align}
  \mathcal{Z}_{\y}[ A_{2}\langle [2]\rangle] & = \mathcal{Z}_{\y}[ \text{diag}(V,id_{n-1}) \hat A_{2}\langle [2]\rangle]  \notag\\
  & =\text{diag}(V,\textit{id}_{n-1})\circ \mathcal{Z}_{\y}[\hat A_{2}\langle [2]\rangle] \notag\\
  & =\text{diag}(V,\textit{id}_{n-1})\circ \CT_{t,2}\CT_{q,1} [\hat{A}].\label{equ-M-M^}
\end{align}
For this iterated constant term, let $\Lambda=(q,t)$, observe that
\begin{align*}
[\hat{A}] &= \frac{\Lambda^{U\b}}{(1-q^{h_{1}} y_1)(1- t^{h_{2}} y_2)\prod_{i=1}^{n-1} (1- \Lambda^{\alpha_{i+2}}y_{i+2}) },
\end{align*}
where $\alpha_{i+2}$ is the $i$-th column of $US_2$. Repeated application of Equation \eqref{equ-addifomula-A} gives
\begin{align}
\CT\limits_{t,2}\CT\limits_{q,1} [\hat{A}]&=\CT\limits_{t}\frac{1}{\underline{1-t^{h_{2}} y_2}}\CT\limits_{q} \frac{1}{\underline{1-q^{h_{1}}y_1}}\cdot\frac{\Lambda^{U\b}}{\prod_{i=1}^{n-1} (1- \Lambda^{\alpha_{i+2}}y_{i+2})}\notag\\
&=\frac{1}{h_1 h_2}\sum_{j_1=0}^{h_1-1}\sum_{j_2=0}^{h_2-1}\frac{\hat{\y}^{U\b\odot (-\mathbf{h})}(\zeta_{1}^{j_1},\zeta_{2}^{j_2})^{U\b}}{\prod_{i=1}^{n-1}(1-y_{i+2}\hat{\y}^{\alpha_{i+2}\odot(-\mathbf{h})}(\zeta_{1}^{j_1},\zeta_{2}^{j_2})^{\alpha_{i+2}})}\notag\\
&=\frac{1}{h_1 h_2}\sum_{j_1=0}^{h_1-1}\sum_{j_2=0}^{h_2-1}\frac{\hat{\y}^{U\b\odot (-\mathbf{h})}\zeta_{2}^{N_{j_1,j_2}}}{\prod_{i=1}^{n-1}(1-y_{i+2}\hat{\y}^{\alpha_{i+2}\odot(-\mathbf{h})}\zeta_{2}^{D_{i,j_1,j_2}})}.\label{equ-unit}
\end{align}
Here, $\odot$ denotes the Hadamard product. We define the vectors $\mathbf{h}=(1/h_1, 1/h_2)$ and $\hat{\y}=(y_1, y_2)$, along with the terms $N_{j_1,j_2} = \frac{h_2}{h_1}j_1(U\b)_{1}+j_2(U\b)_{2}$ and $D_{i,j_1,j_2} = \frac{h_2}{h_1}j_1(\alpha_{i+2})_{1}+j_2(\alpha_{i+2})_{2}$. Furthermore, $\zeta_1 = e^{2\pi\imath/h_1}$ and $\zeta_2 = e^{2\pi\imath/h_2}$ represent primitive $h_1$-th and $h_2$-th roots of unity, respectively, which satisfy the relation $\zeta_1 = \zeta_2^{h_2/h_1}$.

Now with equation \eqref{equ-unibefore}, \eqref{equ-M-M^} and \eqref{equ-unit}, we get
\begin{align}
\langle 1,2\rangle&=\sgn(\det(S_{1}))\text{diag}(V,\textit{id}_{n-1})\circ \CT_{t,2}\CT_{q,1} [\hat{A}]\notag\\
&=\frac{1}{D_{1,2}}\sum_{j_1=0}^{h_1-1}\sum_{j_2=0}^{h_2-1}\frac{\y^{L}
\zeta_{2}^{N_{j_1,j_2}}}{\prod_{i=1}^{n-1}(1-\y^{M_{i}}\zeta_{2}^{D_{i,j_1,j_2}})},\label{equ-12yunit}
\end{align}
which is a sum containing $h_1 h_2=|r_1 p_2 -r_2 p_1|$ terms. Clearly, as $N_{j_1,j_2}$ and $D_{i,j_1,j_2}$ are integers, $\zeta_{2}^{N_{j_1,j_2}}$ and $\zeta_{2}^{D_{i,j_1,j_2}}$ are $h_2$-th roots of unity; thus, they are naturally $|D_{1,2}|$-th roots of unity as well.

This completes the proof.
\end{proof}

Having derived the explicit formula for $\bar{\sigma}_{P(b;\a)}$, it remains to evaluate $\bar{\sigma}_{P(b;\a)}(\mathbf{1})$ via Theorem \ref{theo-GTodd}, following the method from Proposition \ref{prop-polymialwhenbounded}.

Now we can describe the algorithm as follows.
\begin{alg}[\texttt{PolyAL}]\label{alg-HKP}
\mbox{}\\
\textbf{Input:} Coprime positive integers $M, N$; nonnegative integer $b$; linearly independent integer sequences $\mathbf{p} = (p_1, \dots, p_n)$ and $\mathbf{r} = (r_1, \dots, r_n)$ where $p_i M + r_i N > 0$ for all $i$. \\
\textbf{Output:} The value $d(b; p_1 M + r_1 N, \dots, p_n M + r_n N)$.
\begin{enumerate}
\item Determine the decomposition as specified in \eqref{equa-sumij}.
\item Rewrite each term of \eqref{equa-sumij} into the form \eqref{equ-finalformula}.
\item Select an admissible vector $\kappa = (\kappa_1, \dots, \kappa_{n+1})$ and apply the substitution $y_i = e^{\kappa_i s}$.
\item Evaluate the constant term of each substituted expression via Theorem \ref{theo-GTodd}.
\item Sum these constant terms to compute the final output.
\end{enumerate}
\end{alg}

\begin{thm}\label{thm-PolyAL}
Let $n, \Delta \in \Z_{+}$, and let $\a$ be a $\Delta$-modular A-L sequence. That is, there exist coprime positive integers $M$ and $N$ such that $a_i = p_i M + r_i N$, where the integer vectors $\mathbf{p} = (p_1, \dots, p_n)$ and $\mathbf{r} = (r_1, \dots, r_n)$ are linearly independent, and the matrix $\binom{\mathbf{p}}{\mathbf{r}}$ is $\Delta$-modular. Given $M,N,\mathbf p,\mathbf r$ and any nonnegative integer $b$, \texttt{PolyAL} computes $d(b;\a)$ using
\[
O(\log(2+\min\{M,N\}))+O(n^4 \Delta^2 \log n \log \Delta)
\]
arithmetic operations. The first term accounts for a single extended Euclidean computation on $M,N$.
\end{thm}
\begin{proof}
As in the proof of Theorem~\ref{thm-crudealgocomplexity}, Step~1 requires
one extended Euclidean computation on $M,N$. Except this, the most computationally intensive part is Step 4.

Step 1 generates at most $n^2$ terms in \eqref{equa-sumij}, while Step 2 yields $\vert{}D_{i,j}\vert{}$ terms in \eqref{equ-12yunit}, bounded by $\Delta$. 
In Step 4, we apply Theorem \ref{theo-GTodd} to evaluate
$$\CT_{s} \frac{1}{D_{i,j}} \left.\frac{L_{0}\xi_0}{\prod_{k=1}^{n-1}(1-L_{k}\xi_{k})} \right|_{y_{i}=e^{\kappa_{i}s}},$$
where $\xi_0,\dots,\xi_{n-1}$ are $|D_{i,j}|$-th roots of unity.
The analysis is completely analogous to the proof of Proposition \ref{prop-polymialwhenbounded}. Thus, we state the result directly: a single evaluation requires $O(n^2 \Delta \log n \log \Delta)$ operations over $\Q$.
Since Step 4 is executed at most $n^2 \Delta$ times, computing $d(b;\a)$ after the initial extended Euclidean computation takes $O(n^4 \Delta^2 \log n \log \Delta)$ arithmetic operations, with logarithmic factors taken to be at least one. Adding the $O(\log(2+\min\{M,N\}))$ operations for the initial computation gives the stated bound.
\end{proof}

Note that any $B$-bounded A-L sequence $\a$ is naturally $2B^2$-modular. Consequently, after the same extended Euclidean computation, the algorithm in this section uses $\widetilde O(n^4 B^4)$ arithmetic operations, improving the $\widetilde O(n^4 B^6)$ bound in Theorem \ref{thm-crudealgocomplexity}. Furthermore, this algorithm is more general, as instances frequently arise where $\Delta$ remains relatively small even when $|p_{i}|$ and $|r_i|$ are exceptionally large.

\subsection{Applicability of \texttt{PolyAL}}\label{subsec-LLLHKP}

The algorithm developed in the previous subsection naturally raises the following question: for a fixed $\Delta$, how can we determine whether a given sequence $\a$ is a $\Delta$-modular A-L sequence? That is, do there exist coprime integers $M, N$ and integer vectors $\mathbf{p}, \mathbf{r}$ such that $\a = \mathbf{p}M + \mathbf{r}N$, and the matrix $\binom{\mathbf{p}}{\mathbf{r}}$ is $\Delta$-modular? Furthermore, if such parameters exist, how can we compute them efficiently?

Indeed, the problem can be transformed into finding a sequence $\mathbf{k}=(k_1,\dots, k_n)\in \Z^n$ so that $\binom{\a}{\mathbf{k}}$ is $\Delta$-modular.
Assume we are given two coprime positive integers $M$ and $N$. Then there exist integers $u,v$ such that $uM +vN = 1$, or $\det U=1$ where
$U=\begin{pmatrix}
              M & N \\
              -v & u \\
            \end{pmatrix}$ is unimodular.
By denoting
$$T=\binom{\a}{\mathbf{k}} = \begin{pmatrix}
              M & N \\
              -v & u \\
            \end{pmatrix} \binom{\mathbf{p}}{\mathbf{r}} =UP,$$
we see that $T$ and $P$ are unimodularly equivalent, yielding $\operatorname{rank}(T) = \operatorname{rank}(P) = 2$. On the other hand, if we find a $\mathbf{k}$ such that $T$ is $\Delta$-modular,
then $U$ can be any unimodular matrix with positive $M,N$.

\begin{cor}\label{cor-PolyAL-auxiliary-row}
Let $\a\in\Z_{+}^{n}$ and $b\in\Z_{\geq 0}$. Suppose an integer row $\mathbf{k}$ is given such that $\binom{\a}{\mathbf{k}}$ has rank two and all its $2\times2$ minors have absolute value at most $\Delta$. Then $\texttt{PolyAL}$ computes $d(b;\a)$ using $\widetilde O(n^4\Delta^2)$ arithmetic operations in total.
\end{cor}
\begin{proof}
Take $M'=N'=1$, $\mathbf{p}'=\a-\mathbf{k}$, and $\mathbf{r}'=\mathbf{k}$. Then $\a=\mathbf{p}'+\mathbf{r}'$, and for every $i,j$,
\[
r_i' p_j'-r_j' p_i'=k_i(a_j-k_j)-k_j(a_i-k_i)=k_i a_j-k_j a_i.
\]
Thus $\mathbf{p}'$ and $\mathbf{r}'$ are linearly independent and satisfy the same minor bound. Forming them takes $O(n)$ operations, and the initial Euclidean term in Theorem~\ref{thm-PolyAL} is $O(1)$ for $M'=N'=1$.
\end{proof}

We now construct the following matrix to determine $\mathbf{k}$.
\begin{dfn}
Let $E = \{(i,j) \in \mathbb{Z}^2 \mid 1 \le i < j \le n\}$ be the index set of pairs, equipped with the standard lexicographical order. We define $K$ to be the $\binom{n}{2} \times n$ matrix over $\R$ whose rows are indexed by $E$ and whose columns are indexed by $\{1, 2, \dots, n\}$. The entry $K_{(i,j), k}$ in row $(i,j)$ and column $k$ is given by:$$K_{(i,j), k} = \begin{cases} a_j, & \text{if } k = i, \\ -a_i, & \text{if } k = j, \\ 0, & \text{otherwise.} \end{cases}$$Equivalently, letting $\mathbf{e}_1, \dots, \mathbf{e}_n$ denote the standard basis vectors of $\R^n$, the row vector of $K$ corresponding to the index $(i,j)$ is given by $a_j \mathbf{e}_i^\top - a_i \mathbf{e}_j^\top$.
\end{dfn}

Let $\mathbf{k}=(k_1,\dots,k_n)$. By examining the structure of the matrix $K$, we observe that
$$\lVert K \cdot\mathbf{k}^\top \rVert_\infty = \max_{i,j} \vert{}k_i a_j - k_j a_i\vert{}.$$
Therefore, the problem is equivalent to the well-known Shortest Vector Problem (SVP)—namely, computing the shortest nonzero vector in the lattice generated by the columns of $K$ with respect to the $\ell_\infty$-norm.
Denoting this minimum $\ell_\infty$-norm by $\delta$, the answer to our problem is affirmative if $\delta \leq \Delta$, and negative otherwise.

It is well known that SVP is NP-hard, making the exact computation of $\delta$ difficult. A theoretical upper bound can be derived via Minkowski's First Theorem on convex bodies (see Cassels~\cite{C97}); we state without proof that $\delta \leq \lVert \a \rVert_2^{(n-2)/(n-1)}$ for general $\a$. Although numerical experiments on random instances confirm the tightness of this bound, it remains too large for our purposes. Consequently, the proposed algorithm does not offer a substantial advantage for arbitrary sequences $\a$.

Nevertheless, for a specific instance $\a$, the LLL algorithm can be used to efficiently approximate the shortest vector with respect to both the $\ell_2$ and $\ell_\infty$ norms. Suppose $\gamma_1, \gamma_2, \dots$ form the reduced basis for the columns of $K$. By the properties of lattice reduction, each basis vector $\gamma_j$ is relatively short. In particular, we have $\delta \leq \lVert \gamma_1 \rVert_\infty \leq l(n) \delta$, where $l(n)$ is a factor depending solely on the dimension $n$. Consequently, the condition $\lVert \gamma_1 \rVert_\infty \leq \Delta$ serves as a sufficient positive certificate for verifying whether $\delta \leq \Delta$.

\subsection{Computational Experiments and Performance}
\label{sec:experiments}
Two well-known algorithms for solving this class of problems are \texttt{LattE} and \texttt{DecDenu} \cite{xin2024combinatorial}. As noted by Xin, Y.~Zhang, and Z.~Zhang~\cite{xin2024combinatorial}, \texttt{DecDenu} is less efficient than \texttt{LattE} for the instances considered here, which motivates the work presented in this paper. We therefore implement our algorithm as a Maple package and compare its performance with that of \texttt{LattE} on a set of test examples. The package is available at \url{https://github.com/Jinlong-Tang/AL-Prob}.

Our Maple package relies on two primary procedures: \texttt{IsAL} and \texttt{PolyAL}. The function \texttt{IsAL(a, M, N)} takes a sequence of positive integers $\a=[a_1, \dots, a_n]$ alongside coprime positive integers $M$ and $N$ to construct a good A-L representation (see Subsection~\ref{subsec-LLLHKP}). That is, it returns two linearly independent integer vectors $\mathbf{p}$ and $\mathbf{r}$ satisfying $a_i=p_i M + r_i N$ for all $i$. If $|r_ip_j-r_jp_i|$ are small for all $i,j$, then invoking $\texttt{PolyAL}(M,N,\mathbf{p},\mathbf{r},b)$ will efficiently compute the exact value of $d(b; \a)$.

Table~\ref{tab:instances} lists the 19 knapsack test instances, all of the form $a_i = p_i M + r_i N$. All test instances are generated directly from given
A-L representations $(M,N,\mathbf{p},\mathbf{r})$
by setting $a_i=p_iM+r_iN$ for $1\leq i\leq n$.
The first twelve (\texttt{cuww4}, \texttt{cuww5}, \texttt{prob1}--\texttt{prob10}) are listed by Aardal and Lenstra~\cite{AL04}. Because the corresponding vectors $\mathbf{p}$ and $\mathbf{r}$ are not supplied, we recover them via the \texttt{IsAL} procedure. 
To explore the computational limits, we construct seven higher-dimensional instances (\texttt{AL1}--\texttt{AL7}) by uniformly sampling $M \in [10000, 20000]$, $N \in [1000, 2000]$, $p_i \in [1, 10]$, and $r_i \in [-10, 10]$. Finally, since the decomposition employed by the algorithm depends on the target sum $b$ (Theorem~\ref{Thm-Center}), we choose sufficiently large $b$ to keep the experiments general. For the first twelve instances, we multiply the
original target values by $1111$ to make right-hand sides sufficiently large, where
the multiplier $1111$ is a benchmark choice rather
than a requirement of our algorithm.

Table~\ref{tab:comparison} presents the computational results.  Since the outputs of \texttt{PolyAL} and \texttt{LattE} are completely identical for each instance, we omit the specific output values here and focus solely on comparing computation scale and execution time. As expected, the C++-based \texttt{LattE} runs faster on smaller instances (e.g., \texttt{cuww} and \texttt{prob1}--\texttt{prob7}) due to the inherent speed advantage of compiled C++ over interpreted Maple. Conversely, when the input scale increases, \texttt{PolyAL} scales significantly better. The column $\Delta$ in both tables gives $\max_{i<j}|r_ip_j-r_jp_i|$.

The key to this speedup lies in the number of rational functions. Traditional geometry-based algorithms rely on cone decomposition, generating millions of rational functions for extreme instances (e.g., 2,568,690 terms for \texttt{AL7}). By strictly bounding structural complexity, our method evaluates only 7,141 terms for the same instance.

{\footnotesize
\begin{longtable}{@{}l r p{0.55\textwidth} l@{}}
\caption{Input data for the test problems.}
\label{tab:instances}\\
\toprule
Problem ID & $\Delta$ & $\a$ & $b$ \\
\midrule
\endfirsthead
\toprule
Problem ID & $\Delta$ & $\a$ & $b$ \\
\midrule
\endhead

\texttt{cuww4} & 24 & 13211, 13212, 39638, 52844, 66060, 79268, 92482
   & $104723595\times 1111$ \\
\texttt{cuww5} & 32 & 13429, 26850, 26855, 40280, 40281, 53711, 53714, 67141
   & $45094583\times 1111$ \\
\texttt{prob1} & 153 & 25067, 49300, 49717, 62124, 87608, 88025, 113673, 119169
   & $33367335\times 1111$ \\
\texttt{prob2} & 98 & 11948, 23330, 30635, 44197, 92754, 123389, 136951, 140745
   & $14215206\times 1111$ \\
\texttt{prob3} & 121 & 39559, 61679, 79625, 99658, 133404, 137071, 159757, 173977
   & $58424799\times 1111$ \\
\texttt{prob4} & 70 & 48709, 55893, 62177, 65919, 86271, 87692, 102881, 109765
   & $60575665\times 1111$ \\
\texttt{prob5} & 162 & 28637, 48198, 80330, 91980, 102221, 135518, 165564, 176049
   & $62442884\times 1111$ \\
\texttt{prob6} & 86 & 20601, 40429, 40429, 45415, 53725, 61919, 64470, 69340, 78539, 95043
   & $22382774\times 1111$ \\
\texttt{prob7} & 80 & 18902, 26720, 34538, 34868, 49201, 49531, 65167, 66800, 84069, 137179
   & $27267751\times 1111$ \\
\texttt{prob8} & 95 & 17035, 45529, 48317, 48506, 86120, 100178, 112464, 115819, 125128, 129688
   & $21733990\times 1111$ \\
\texttt{prob9} & 126 & 13719, 20289, 29067, 60517, 64354, 65633, 76969, 102024, 106036, 119930
   & $13385099\times 1111$ \\
\texttt{prob10} & 108 & 45276, 70778, 86911, 92634, 97839, 125941, 134269, 141033, 147279, 153525
   & $106925261\times 1111$ \\
\texttt{AL1} & 120 & 63496, 83188, 7669, 110013, 62391, 24079, 130810, 48158, 4387, 92498, 207970
   & 12312223212331 \\
\texttt{AL2} & 108 & 186637, 138479, 38312, 134092, 30107, 90857, 72237, 103985, 134092, 92498, 91393, 101808
   & 12312223212331 \\
\texttt{AL3} & 161 & 183060, 156714, 189760, 63509, 176814, 135847, 128693, 178489, 14621, 176360, 24217, 85284, 114072
   & 12312223212331 \\
\texttt{AL4} & 154 & 104503, 31047, 40665, 90217, 51886, 106837, 148374, 61504, 11952, 8746, 34253, 77393, 78996, 164404
   & 12312223212331 \\
\texttt{AL5} & 116 & 106106, 164404, 95757, 41537, 67044, 86139, 74187, 19967, 122867, 72584, 154786, 38331, 14286, 52617, 124470
   & 12312223212331 \\
\texttt{AL6} & 162 & 29266, 87837, 115672, 55180, 47986, 71979, 105577, 166069, 41282, 73900, 130550, 51828, 101735, 106067, 145428, 47986
   & 12312223212331 \\
\texttt{AL7} & 181 & 147113, 149052, 42220, 135917, 61172, 48037, 141515, 70648, 149271, 97356, 78623, 57513, 135479, 157027, 149052, 57294, 124064
   & 12312223212331 \\
\bottomrule
\end{longtable}
}

\begin{table}[H]
\centering
\caption{Comparison of computation scale and execution time between \texttt{PolyAL} and \texttt{LattE}.}
\footnotesize
\label{tab:comparison}
\begin{tabular}{l r rr rr}
\toprule
\multirow{2}{*}{Problem ID} & \multirow{2}{*}{$\Delta$}
& \multicolumn{2}{c}{\texttt{PolyAL} (Maple)}
& \multicolumn{2}{c}{\texttt{LattE} (C++)} \\
\cmidrule(lr){3-4}
\cmidrule(lr){5-6}
& & Rational functions & Time (s)
& Rational functions & Time (s) \\
\midrule
\texttt{cuww4} & 24 & 198  & 0.593   & 364     & 0.04    \\
\texttt{cuww5} & 32 & 307  & 0.575   & 2514    & 0.30    \\
\texttt{prob1} & 153 & 1465 & 11.159  & 10618   & 1.51    \\
\texttt{prob2} & 98 & 1164 & 3.318   & 6244    & 0.94    \\
\texttt{prob3} & 121 & 1570 & 17.201  & 12972   & 1.886   \\
\texttt{prob4} & 70 & 1016 & 2.234   & 9732    & 1.33    \\
\texttt{prob5} & 162 & 1442 & 4.175   & 8414    & 1.28    \\
\texttt{prob6} & 86 & 1537 & 13.732  & 26448   & 5.81    \\
\texttt{prob7} & 80 & 1294 & 6.765   & 20192   & 4.76    \\
\texttt{prob8} & 95 & 2197 & 15.801  & 62044   & 14.76   \\
\texttt{prob9} & 126 & 1563 & 6.261   & 36354   & 8.55    \\
\texttt{prob10} & 108 & 1781 & 5.993   & 38638   & 8.25    \\
\texttt{AL1} & 120 & 1802 & 8.899   & 68264   & 19.83   \\
\texttt{AL2} & 108 & 2483 & 14.760  & 105092  & 32.55   \\
\texttt{AL3} & 161 & 4599 & 78.921  & 441384  & 181.91  \\
\texttt{AL4} & 154 & 3507 & 76.813  & 461936  & 221.13  \\
\texttt{AL5} & 116 & 3426 & 17.469  & 409754  & 246.33  \\
\texttt{AL6} & 162 & 6308 & 231.193 & 1925320 & 1368.31 \\
\texttt{AL7} & 181 & 7141 & 707.071 & 2568690 & 2205.53 \\
\bottomrule
\end{tabular}
\end{table}

\section{Concluding Remarks}\label{remark}
Consider a polytope of the form $P(A,\b) = \{\alpha \in \R_{\geq 0}^n : A\alpha = \b\}$, where $A$ has $r$ rows. In general, the lattice point counting problem for such polytopes is $\#P$-hard. This motivates the search for classes in which $\sigma_P(\mathbf{1})$ can be evaluated efficiently. Two known cases are the following. The first is when the dimension $n-r$ is fixed, solvable via Barvinok's algorithm \cite{barvinok1994polynomial}; the second is when $\Delta$ and $r$ are fixed and the matrix $A$ is $\Delta$-modular, solvable via the $\Delta$-modular algorithm of Gribanov and Zolotykh \cite{GribanovZolotykh2022}.

The arithmetic bounds in Sections \ref{sec-3}--\ref{sec-5} motivate the study of broader classes. Specifically, we define a matrix $A$ to be \emph{almost (resp., nearly) $\Delta$-modular} if removing several columns (resp., one column) from $A$ yields a $\Delta$-modular matrix. We have resolved the case where $r=1$ and $A$ is nearly $\Delta$-modular in Section \ref{sec-3}, and the techniques developed in Section \ref{sec-5} can be applied to the $r=2$ case. This framework should naturally extend to any fixed $r$ when $A$ is nearly $\Delta$-modular. Furthermore, we have begun investigating the case where $r=1$ and $A$ is almost $\Delta$-modular \cite{XinZZ2026}, and we ultimately aim to address the general case—that is, when $r$ is fixed and $A$ is almost $\Delta$-modular.

\noindent
{\small \textbf{Acknowledgements:}
This work is partially supported by the National Natural Science
Foundation of China [12571355].
}


\begin{thebibliography}{99}

\bibitem{AL04}
Karen Aardal and Arjen~K. Lenstra.
\newblock Hard equality constrained integer knapsacks.
\newblock {\em Math. Oper. Res.}, 29(3):724--738, 2004.

\bibitem{AardalScavuzzoWolsey2023}
Karen Aardal, Lara Scavuzzo, and L.~Wolsey.
\newblock A study of lattice reformulations for integer programming.
\newblock {\em Oper. Res. Lett.}, 51(4):401--407, 2023.

\bibitem{BBDDKV15}
Velleda Baldoni, Nicole Berline, J.~A. De~Loera, Brandon~E. Dutra, Matthias
  K{\"o}ppe, and Mich{\`e}le Vergne.
\newblock Coefficients of {Sylvester}'s denumerant.
\newblock {\em Integers}, 15:Paper A11, 2015.

\bibitem{barvinok1994polynomial}
Alexander Barvinok.
\newblock A polynomial time algorithm for counting integral points in polyhedra
  when the dimension is fixed.
\newblock {\em Math. Oper. Res.}, 19(4):769--779, 1994.

\bibitem{BP99}
Alexander Barvinok and James~E. Pommersheim.
\newblock An algorithmic theory of lattice points in polyhedra.
\newblock In {\em New perspectives in algebraic combinatorics}, pages 91--147.
  Cambridge: Cambridge University Press, 1999.

\bibitem{C97}
J.~W.~S. Cassels.
\newblock {\em An introduction to the geometry of numbers}.
\newblock Berlin: Springer, 1997.

\bibitem{CUWW97}
G.~Cornu{\'e}jols, R.~Urbaniak, R.~Weismantel, and L.~Wolsey.
\newblock Decomposition of integer programs and of generating sets.
\newblock In {\em Algorithms -- ESA '97. 5th annual European symposium, Graz,
  Austria, September 15--17, 1997. Proceedings}, pages 92--103. Berlin:
  Springer, 1997.
  
\bibitem{DGKKKKM2026}
G. Dakhno, D.~V. Gribanov, N. Kasianov, A. Kats, A. Kupavskii, N. Kuz'min and S. Moiseev.
\newblock Polynomial and {Hyperplane} {Avoidance} over the {Integers} with {Applications} to {Integer}-{Point} {Counting}.
\newblock Preprint, {arXiv}:2411.07030 [cs.{CC}] (2026), 2026.

\bibitem{DeLoeraHawsHemmeckeHugginsYoshida2005}
J.~A. De~Loera, D.~Haws, R.~Hemmecke, P.~Huggins, and R.~Yoshida.
\newblock A computational study of integer programming algorithms based on
  {Barvinok}'s rational functions.
\newblock {\em Discrete Optim.}, 2(2):135--144, 2005.

\bibitem{DHK13}
J.~A. De~Loera, R.~Hemmecke, and Matthias K{\"o}ppe.
\newblock {\em Algebraic and geometric ideas in the theory of discrete
  optimization}, volume~14 of {\em MOS-SIAM Ser. Optim.}
\newblock Philadelphia, PA: Society for Industrial {and} Applied Mathematics
  (SIAM), 2013.

\bibitem{DeLoeraHemmeckeTauzerYoshida2004}
J.~A. De~Loera, R.~Hemmecke, Jeremiah Tauzer, and R.~Yoshida.
\newblock Effective lattice point counting in rational convex polytopes.
\newblock {\em J. Symb. Comput.}, 38(4):1273--1302, 2004.

\bibitem{GSM2023}
D.~V. Gribanov and D.~S. Malyshev. 
\newblock A faster algorithm for counting the integer points number in {{\(\Delta \)}}-modular polyhedra. 
\newblock {\em Sib. {\`E}lektron. Mat. Izv.} 19 (2022), 613–626. Corrected version
with I.~A. Shumilov (2023).

\bibitem{GSMZ2024}
D.~V. Gribanov, I.~A. Shumilov, D.~S. Malyshev and N.~Yu. Zolotykh.
\newblock Faster algorithms for sparse {ILP} and hypergraph multi-packing/multi-cover problems
\newblock {\em J. Glob. Optim.} 89 (2024), 1033–1067.

\bibitem{GribanovZolotykh2022}
D.~V. Gribanov and N.~Yu. Zolotykh.
\newblock On lattice point counting in {{\(\varDelta\)}}-modular polyhedra.
\newblock {\em Optim. Lett.}, 16(7):1991--2018, 2022.

\bibitem{GLS93}
Martin Gr{\"o}tschel, L{\'a}szl{\'o} Lov{\'a}sz, and Alexander Schrijver.
\newblock {\em Geometric algorithms and combinatorial optimization}.
\newblock Berlin: Springer-Verlag, 1993.

\bibitem{KrishnamoorthyPataki2009}
Bala Krishnamoorthy and G{\'a}bor Pataki.
\newblock Column basis reduction and decomposable knapsack problems.
\newblock {\em Discrete Optim.}, 6(3):242--270, 2009.

\bibitem{LLL82}
Arjen~K. Lenstra, Hendrik~W. Lenstra, and L{\'a}szl{\'o} Lov{\'a}sz.
\newblock Factoring polynomials with rational coefficients.
\newblock {\em Math. Ann.}, 261:515--534, 1982.

\bibitem{S98}
Alexander Schrijver.
\newblock {\em Theory of linear and integer programming}.
\newblock Chichester: Wiley, 1998.

\bibitem{V13}
Joachim von~zur Gathen and J{\"u}rgen Gerhard.
\newblock {\em Modern computer algebra}.
\newblock Cambridge: Cambridge University Press, 2013.

\bibitem{X04}
Guoce Xin.
\newblock A fast algorithm for {MacMahon}'s partition analysis.
\newblock {\em Electron. J. Comb.}, 11(1) (2004), Research Paper R58.

\bibitem{X15}
Guoce Xin.
\newblock A {Euclid} style algorithm for {MacMahon}'s partition analysis.
\newblock {\em J. Comb. Theory, Ser. A}, 131:32--60, 2015.


\bibitem{XXZ25}
Guoce Xin, Xinyu Xu, and Zihao Zhang.
\newblock A combinatorial simplicial cone decomposition.
\newblock Preprint, {arXiv}:2501.06691 [math.{CO}] (2025), 2025.

\bibitem{XinZhang2023}
Guoce Xin and Chen Zhang.
\newblock An algebraic combinatorial approach to {Sylvester}'s denumerant.
\newblock {\em Ramanujan J.}, Volume 66, Article 64 (2025).

\bibitem{XZ25}
Guoce Xin and Chen Zhang.
\newblock A polynomial time algorithm for {Sylvester} waves when entries are
  bounded.
\newblock {\em Adv. Appl. Math.}, 170 (2025), 102931.

\bibitem{XinZZ2026}
Guoce Xin, Chen Zhang, and Zihao Zhang.
\newblock A polynomial time algorithm for almost bounded denumerant.
\newblock Preprint, {arXiv}:2608.14985 [math.{CO}] (2026), 2026.

\bibitem{xin2024combinatorial}
Guoce Xin, Yingrui Zhang, and Zihao Zhang.
\newblock A {Combinatorial} {Decomposition} of {Knapsack} {Cones}.
\newblock Preprint, {arXiv}:2406.13974 [math.{CO}] (2024), 2024.

\bibitem{XZZ25}
Guoce Xin, Yingrui Zhang, and Zihao Zhang.
\newblock Fast evaluation of generalized {Todd} polynomials: applications to
  {MacMahon}'s partition analysis and integer programming.
\newblock {\em J. Symb. Comput.}, 130 (2025), 102420.
\end{thebibliography}
\end{document}